\documentclass{amsart}
\usepackage{latexsym,amsxtra,amscd,ifthen}
\usepackage{amsfonts}
\usepackage{verbatim}
\usepackage{amsmath}
\usepackage{amsthm}
\usepackage{amssymb}
\usepackage{url}

\theoremstyle{plain}
\newtheorem{maintheorem}{Theorem}

\newtheorem{theorem}{Theorem}
\newtheorem{lemma}[theorem]{Lemma}
\newtheorem{proposition}[theorem]{Proposition}
\newtheorem{corollary}[theorem]{Corollary}
\newtheorem{conjecture}[theorem]{Conjecture}

\numberwithin{theorem}{section}
\numberwithin{equation}{theorem}

\theoremstyle{definition}
\newtheorem{definition}[theorem]{Definition}
\newtheorem{example}[theorem]{Example}
\newtheorem{remark}[theorem]{Remark}

\newtheorem*{question*}{Question}

\newcommand{\cwlt}{(\textup{cwlt})}
\newcommand{\af}{\textup{af}}
\newcommand{\Af}{\textup{Af}}
\newcommand{\Aff}{\textup{A}}

\newcommand{\INT}{\textup{int}}
\newcommand{\Q}{\mathbb{Q}}
\newcommand{\Z}{\mathbb{Z}}
\newcommand{\N}{\mathbb{N}}
\newcommand{\xxMD}{M\!D}
\newcommand{\MD}[2]{\xxMD_{#1}(#2:\tr)}
\DeclareMathOperator{\ch}{char}
\DeclareMathOperator{\rk}{rk}

\DeclareMathOperator{\End}{End}

\DeclareMathOperator{\LNDer}{LNDer}

\DeclareMathOperator{\Aut}{Aut}
\DeclareMathOperator{\gr}{gr}
\DeclareMathOperator{\tr}{tr}
\DeclareMathOperator{\im}{im}

\newcommand{\uni}{\textup{uni}}

\begin{document}

\title[Discriminant criterion and automorphism groups]
{The discriminant criterion and \\automorphism groups of quantized algebras}

\author{S. Ceken, J. H. Palmieri, Y.-H. Wang and J. J. Zhang}

\address{Ceken: Department of Mathematics, Akdeniz University, 07058 Antalya,
Turkey}

\email{secilceken@akdeniz.edu.tr}

\address{Palmieri: Department of Mathematics, Box 354350,
University of Washington, Seattle, Washington 98195, USA}

\email{palmieri@math.washington.edu}

\address{Wang: Department of Applied Mathematics,
Shanghai University of Finance and
Economics, Shanghai 200433, China}

\email{yhw@mail.shufe.edu.cn}

\address{Zhang: Department of Mathematics, Box 354350,
University of Washington, Seattle, Washington 98195, USA}

\email{zhang@math.washington.edu}

\begin{abstract}
We compute the automorphism groups of some quantized algebras,
including tensor products of quantum Weyl algebras and some skew
polynomial rings.
\end{abstract}





\maketitle


\section*{Introduction}

It is well-known that every automorphism of the polynomial ring
$k[x]$, where $k$ is a field, is determined by the assignment
$x\mapsto ax+b$ for some $a\in k^\times:=k\setminus \{0\}$ and $b\in
k$. Every automorphism of $k[x_1,x_2]$ is \emph{tame}, that is, it is
generated by affine and elementary automorphisms (defined below). This
result was first proved by Jung \cite{Ju} in 1942 for characteristic
zero and then by van der Kulk \cite{vdK} in 1953 for arbitrary
characteristic. A structure theorem for the automorphism group
of $k[x_1,x_2]$ was also given in \cite{vdK}.  The automorphism
group of $k[x_1,x_2,x_3]$ has not yet been fully understood, and the
best result in this direction is the existence of wild automorphisms
(e.g.~the Nagata automorphism) by Shestakov-Umirbaev \cite{SU}.

The automorphism group of the skew polynomial ring
$k_q[x_1,\dots,x_n]$, where $q\in k^\times$ is not a root of unity and
$n\geq 2$, was completely described by Alev and Chamarie \cite[Theorem
1.4.6]{AlC} in 1992.  Since then, many researchers have been
successfully computing the automorphism groups of classes of
interesting infinite-dimensional noncommutative algebras, including
certain quantum groups, generalized quantum Weyl algebras, skew
polynomial rings and many more -- see \cite{AlC, AlD, AnD, BJ, GTK,
SAV, Y1, Y2}, among others. In particular, Yakimov has proved the
Andruskiewitsch-Dumas conjecture and the Launois-Lenagan conjecture by
using a rigidity theorem for quantum tori, see \cite{Y1, Y2}, each of
which determines the automorphism group of a family of quantized
algebras with parameter $q$ being not a root of unity. See also
\cite{GY} for a uniform approach to these two conjectures.

Determining the automorphism group of an algebra is generally a
very difficult problem. In \cite{CPWZ} we introduced the discriminant
method to compute automorphism groups of some noncommutative algebras.
In this paper we continue to develop new methods and extend ideas
from \cite{CPWZ} for both discriminants and automorphism groups.

Suppose $A$ is a filtered algebra with filtration
$\{F_i A\}_{i\geq 0}$ such that the associated graded algebra $\gr A$
is generated in degree 1. An automorphism $g$ of $A$ is
\emph{affine} if $g(F_1 A)\subset F_1 A$.  An
automorphism $h$ of the polynomial extension $A[t]$
is called \emph{triangular} if there is a $g\in \Aut(A)$,
$c\in k^\times$ and  $r$ in the center of $A$ such that
\[
h(t)=ct+r \quad {\text{and}}\quad
h(x)=g(x)\in A \quad {\text{for all $x\in A$}}.
\]
As in \cite{CPWZ}, we use the discriminant to control
automorphisms and locally nilpotent derivations.
Let $C(A)$ denote the center of $A$.
Here is the discriminant criterion for affine
automorphisms.

\begin{maintheorem}
\label{yythm0.1}
Assume $k$ is a field of characteristic $0$.  Let $A$ be a filtered
algebra, finite over its center, such that the associated graded ring
$\gr A$ is a connected graded domain. Suppose that the
$v$-discriminant $d_v(A/C(A))$ is dominating
for some $v \geq 1$. Then the following hold.
\begin{enumerate}
\item
Every automorphism of $A$ is affine, and $\Aut(A)$ is
an algebraic group that fits into the exact
sequence
\begin{equation}
\label{0.1.1}\tag{*}
 1\to (k^\times)^r\to \Aut(A)\to S\to 1,
\end{equation}
where $r\geq 0$ and $S$ is a finite group. Indeed, $\Aut(A)=S\ltimes
(k^\times)^r$.
\item
Every automorphism of the polynomial extension $A[t]$
is triangular.
\item
Every locally nilpotent derivation of $A$ is zero.
\end{enumerate}
\end{maintheorem}

The terminology will be explained in Section 1. This is proved below
(in slightly more general form) as Theorem \ref{yythm1.13}.

The discriminant criterion is very effective in computing the
automorphism group for a large class of noncommutative algebras
(examples can be found in \cite{CPWZ} and in this paper), but the
computation of the discriminant can be difficult.  It would be nice to
develop new theories and efficient computational tools for the
discriminant in the setting of noncommutative algebra.

In this paper we apply our methods to two families of
quantized algebras: quantum Weyl algebras and skew polynomial
rings. We recall these next.

Let $q$ be a nonzero scalar in $k$ and let $A_q$ be the
\emph{$q$-quantum Weyl algebra}, the algebra generated by $x$ and $y$
subject to the relation $yx=qxy+1$ (we assume that $q\neq 1$, but $q$
need not be a root of unity). Consider the tensor product
$B:=A_{q_1}\otimes \cdots \otimes A_{q_m}$ of quantum Weyl algebras,
where $q_i\in k^\times \setminus \{1\}$ for all $i$.  Since we are not
assuming that the $q_i$ are roots of unity, $B$ need not be finite
over its center and so
the hypotheses of Theorem \ref{yythm0.1} might fail; however, the
conclusions hold.

\begin{maintheorem}
\label{yythm0.2} Let $k$ be a field.
Let $B=A_{q_1}\otimes \cdots \otimes A_{q_m}$ and
assume that $q_i\neq 1$ for all $i=1,\dots,m$. Then the following hold.
\begin{enumerate}
\item 
Every automorphism of $B$ is affine, and $\Aut(B)$ is an algebraic
group that fits into an exact sequence of the form \eqref{0.1.1}, with
$r=m$.
\item
The automorphism group of $B[t]$ is triangular.
\item
If $\ch k=0$, then every locally nilpotent derivation of $B$ is zero.
\end{enumerate}
\end{maintheorem}

See Section \ref{yysec5} for the proof. 
As a consequence of Theorem \ref{yythm0.2},
the following hold [Theorem \ref{yythm5.7}]:
\begin{itemize}
\item
If $q_i\neq \pm1$ and $q_i\neq q_j^{\pm 1}$ for all $i\neq j$, then
$\Aut(B)=(k^\times)^m$.
\item
If $q_i=q\neq \pm 1$ for all $i$, then $\Aut(B)=S_m \ltimes (k^\times)^m$.
\end{itemize}

Let $\{p_{ij}\in k^\times\mid 1\leq i<j\leq n\}$ be a set of
parameters, and set $p_{ji}=p_{ij}^{-1}$ and $p_{ii}=p_{jj}=1$ for all
$i<j$. In this paper, a \emph{skew polynomial ring} is defined to be
the algebra generated by $x_1,\dots,x_n$ subject to the relations $x_j
x_i=p_{ij} x_i x_j$ for all $i<j$, and is denoted by
$k_{p_{ij}}[x_1,\dots,x_n]$.  Recall from \cite[Chapter 13]{MR} that a
\emph{PI algebra} is one which satisfies a \emph{polynomial
identity}. Skew polynomial rings are PI if and only if they are finite
over their center; hence the skew polynomial ring
$k_{p_{ij}}[x_1,\dots,x_n]$ is PI if and only if each $p_{ij}$ is a
root of unity.  The automorphism groups of skew polynomial rings have
been studied by several authors \cite{AlC, Y1}. The next result says
that the discriminant criterion works well for PI skew polynomial
rings.

\begin{maintheorem}[Theorem \ref{yythm3.1}]
\label{yythm0.3}
Let $A=k_{p_{ij}}[x_1,\dots,x_n]$ be a PI skew polynomial ring over
the commutative domain $k$. Then the following are equivalent.
\begin{enumerate}
\item
$d_w(A/C(A))$ is dominating, where $w=\rk(A/C(A))$.
\item
Every automorphism of $A$ is affine.
\item
Every automorphism of $A[t]$ is triangular.
\item
$C(A) \subset k\langle x_1^{\alpha_1},\dots, x_n^{\alpha_n}
\rangle$ for some $\alpha_1,\dots, \alpha_n\geq 2$.
\end{enumerate}
If $\Z \subset k$, then the above are also equivalent to
\begin{enumerate}
\setcounter{enumi}{4}
\item
Every locally nilpotent derivation is zero.
\end{enumerate}
\end{maintheorem}

Note that the implication (1) $\Rightarrow$ (5) fails when $\ch k \neq 0$
\cite[Example 3.9]{CPWZ}.

One example is $k_q[x_1,\dots,x_n]$ with $n$ even and
$q\neq 1$ a primitive $\ell$th root of unity. In this case,
$C(A)=k[x_1^{\ell},\dots,x_n^{\ell}]$, so part (4) of the above holds.
Therefore all of (1)--(5) hold. By part (2),
$\Aut(k_q[x_1,\dots,x_n])$ is affine. An easy computation shows that
\begin{equation}
\label{0.3.1}\tag{0.3.1}
\Aut(k_q[x_1,\dots,x_n])=\begin{cases} (k^\times)^n & {\text{if}}\quad
q\neq \pm 1,\\
S_n \ltimes (k^\times)^n & {\text{if}}\quad q=-1.\end{cases}
\end{equation}
If $n$ is odd and $q$ is a root of unity, then
$\Aut(k_q[x_1,\dots,x_n])$ is not affine -- see Example \ref{yyex1.8}
-- and is much more
complicated. The structure of
$\Aut(k_q[x_1,\dots,x_n])$ is not well understood for $n$ odd, even
when $n=3$.

We have some results concerning automorphisms of not necessarily PI
skew polynomial rings. We need to introduce some notation.  For any
$1\leq s\leq n$, let
\[
T_s=\{(d_1,\dots, \widehat{d}_s,\dots, d_n)\in \N^{n-1}
\mid \prod_{\substack{j=1\\j\neq s}}^n p_{ij}^{d_j}=p_{is} \; \forall \;
i\neq s\}.
\]
We show in Theorem \ref{yythm3.8} that in the PI case,
if $T_s=\emptyset$ for all $s$, then every automorphism of
$A$ is affine. Note also that in the PI case, if $T_s$ is nonempty, then
$T_s$ is in fact infinite.  If we drop the PI assumption and
we allow at most one $T_s$ to be infinite, we can still understand the
automorphism group, as described in the next result.

An automorphism $g$ of $k_{p_{ij}}[x_1,\dots,x_n]$ is called
\emph{elementary} if there is an $s$ and an element $f$
generated by $x_1,\dots, \widehat{x}_s,\dots,x_n$ such that
\[
g(x_i)=\begin{cases} x_i & i\neq s\\ x_s+f & i=s.\end{cases}
\]
An automorphism of $k_{p_{ij}}[x_1,\dots,x_n]$ is called \emph{tame}
if it is generated by affine and elementary automorphisms.

\begin{maintheorem}
\label{yythm0.4} Let $A=k_{p_{ij}}[x_1,\dots,x_n]$ be a (not
necessarily PI) skew polynomial algebra over the commutative domain
$k$, and suppose that $x_i$ is not
central in $A$ for all $i$.  Let $s_0$ be some integer between $1$ and
$n$. Suppose that $T_s$ is finite for all $s\neq s_0$. Then every
automorphism of $A$ is tame.
\end{maintheorem}

This is proved as a consequence of Theorem \ref{yythm3.11}.

The paper is laid out as follows. In Section~\ref{yysec1}, we
introduce the notion of the discriminant and prove Theorem
\ref{yythm0.1} -- note that this result can be viewed as a
generalization of \cite[Theorem 3]{CPWZ}. In Section~\ref{yysec2}, we
compute the discriminants of skew polynomial rings over their
center.  In Section~\ref{yysec3}, we prove that
$\Aut(k_{p_{ij}}[x_1,\dots,x_n])$ is affine if and only if
the discriminant is dominating and then prove Theorems \ref{yythm0.3} and
\ref{yythm0.4}.  We discuss some properties of automorphisms and
discriminants in Section~\ref{yysec4}. In the final
section, we prove Theorem~\ref{yythm0.2}.

\section{The discriminant controls automorphisms}
\label{yysec1}

Throughout the rest of the paper let $k$ be a commutative domain,
and sometimes we further assume that $k$ is a field. Modules,
vector spaces, algebras, tensor products, and morphisms are over
$k$. All algebras are associative with unit.

The beginning of this section overlaps with the paper \cite{CPWZ}.  We
start by recalling the concept of the discriminant in the
noncommutative setting. Let $R$ be a commutative algebra and let $B$
and $F$ be algebras both of which contain $R$ as a subalgebra. In
applications, $F$ would be either $R$ or a ring of fractions of
$R$. An $R$-linear map $\tr: B\to F$ is called a \emph{trace map} if
$\tr(ab)=\tr(ba)$ for all $a,b\in B$.

If $B$ is the $w\times w$-matrix algebra $M_w(R)$ over $R$, we have the
internal trace $\tr_{\INT}: B\to R$ defined
to be the usual matrix trace, namely, $\tr_{\INT}((r_{ij}))=
\sum_{i=1}^w r_{ii}$.
Let $B$ be an $R$-algebra, let $F$ be a localization of
$R$, and suppose that $B_F:=B\otimes_R F$ is
finitely generated free over $F$. Then left
multiplication defines a natural embedding of $R$-algebras $lm:B\to
B_F\to \End_F(B_F)\cong M_w(F)$, where $w$ is the rank $\rk(B_F/F)$.
Then we define the \emph{regular trace map}
by composing:
\[
\tr_{\textup{reg}}: B\xrightarrow{lm} M_w(F)\xrightarrow{\tr_{\INT}} F.
\]
Usually we use the regular trace even if other trace maps
exist. The following definition is well-known; see Reiner's book \cite{Re}.
Let $R^\times$ denote the set of invertible elements in $R$.
If $f,g\in R$ and $f=cg$ for some $c\in R^\times$, then we write
$f=_{R^\times} g$.

\begin{definition} \cite[Definition 1.3]{CPWZ}
\label{yydef1.1} Let $\tr: B\to F$ be a trace map and $v$ be a fixed
integer. Let $Z:=\{z_i\}_{i=1}^v$ be a subset of $B$.
\begin{enumerate}
\item
The \emph{discriminant} of $Z$ is defined to be
\[
d_v(Z:\tr)=\det(\tr(z_iz_j))_{v\times v}\in F.
\]
\item
\cite[Section 10, p.~126]{Re}.
The \emph{$v$-discriminant ideal} (or \emph{$v$-discriminant
$R$-module}) $D_v(B:\tr)$ is the $R$-submodule of $F$ generated by
the set of elements $d_v(Z:\tr)$ for all $Z=\{z_i\}_{i=1}^v\subset
B$.
\item
Suppose $B$ is an $R$-algebra which is finitely generated free over
$R$ of rank $w$. In this case, we take $F=R$. The \emph{discriminant}
of $B$ over $R$ is defined to be
\[
d(B/R)=_{R^\times} d_w(Z:\tr),
\]
where $Z$ is an $R$-basis of $B$.  Note that $d(B/R)$ is well-defined
up to a scalar in $R^\times$ \cite[p.~66, Exer 4.13]{Re}.
\end{enumerate}
\end{definition}

We refer to the books \cite{AW, Re, St} for the classical definition
of discriminant and its connection with the above definition.

To cover a larger class of algebras, in
particular those that are not free over their center, we need a
modified version of the discriminant. Let $B$ be a domain. A normal
element $x \in B$ \emph{divides} $y \in B$ if $y=wx$ for some $w \in
B$.  If $\mathcal{D}:=\{d_i\}_{i\in I}$ is a set of elements in $B$, a
normal element $x\in B$ is called a \emph{common divisor} of
$\mathcal{D}$ if $x$ divides $d_i$ for all $i\in I$. We say a normal
element $x\in B$ is the \emph{greatest common divisor} or \emph{gcd}
of $\mathcal{D}$, denoted by $\gcd \mathcal{D}$, if
\begin{enumerate}
\item
$x$ is a common divisor of $\mathcal{D}$, and
\item
any common divisor $y$ of $\mathcal{D}$ divides $x$.
\end{enumerate}
It follows from part (2) that the gcd of any subset $\mathcal{D}
\subset B$ (if it exists) is unique up to a scalar in $B^\times$.

Note that the gcd in $B$ may be different from the gcd in $R$, if both
exist. For example, the gcd in $R$ could be 1 while the gcd in $B$ is
non-trivial. By definition, the gcd in $R$ is a divisor of the gcd in
$B$. Of course, the gcd in $B$ may be more difficult to
compute since $B$ is typically noncommutative.

\begin{definition}
\label{yydef1.2} Let $\tr: B\to R$ be a trace map and $v$ a
positive integer. Let $Z=\{z_i\}_{i=1}^v$ and $Z'=\{z'_i\}_{i=1}^v$ be
$v$-element subsets of $B$.
\begin{enumerate}
\item
The \emph{discriminant} of the pair $(Z,Z')$ is defined to be
\[
d_v(Z,Z':\tr)=\det(\tr(z_iz'_j))_{v\times v}\in R.
\]
\item
The \emph{modified $v$-discriminant ideal} $\MD{v}{B}$ is
the ideal of $R$ generated by
the set of elements $d_v(Z,Z':\tr)$ for all $Z, Z' \subset B$.
\item
The \emph{$v$-discriminant} $d_v(B/R)$ is defined to be the gcd
in $B$
of the elements $d_v(Z,Z':\tr)$ for all $Z, Z' \subset B$. Equivalently,
the $v$-discriminant $d_v(B/R)$ is the gcd
in $B$ of the elements in $\MD{v}{B}$.
\end{enumerate}
\end{definition}

If $d_v(B/R)$ exists, then the ideal $(d_v(B/R))$ of $B$ generated by
$d_v(B/R)$ is the smallest principal ideal of $B$ which is generated
by a normal element and contains $\MD{v}{B}B$.

It is clear that $D_v(B:\tr)\subset \MD{v}{B}$. Equality should
hold under reasonable hypotheses. For example, if $B$ is an
$R$-algebra which is finitely generated free over $R$ and if
$w=\rk(B/R)$, then $\MD{w}{B}$ equals $D_w(B:\tr)$, both of which
are generated by the single element $d(B/R)$. In this case it is also
true that $d(B/R) =_{B^\times} d_w(B/R)$.
This follows from \eqref{1.10.2},
which states that if $Z$ and $Z'$ are two $R$-bases of $B$, then
\[
d(B/R)=_{R^\times} d_w(Z,Z':\tr).
\]
If $v_1 < v_2$ and if $d_{v_1}(B/R)$ and $d_{v_2}(B/R)$ exist, then
$d_{v_1}(B/R)$ divides $d_{v_2}(B/R)$, by Lemma~\ref{yylem1.4}(5), and
if $v>\rk(B/R)$, then $d_v(B/R)=0$ [Lemma \ref{yylem1.9}(2)].

If $B$ is not free as an $R$-module, then to use
Definition~\ref{yydef1.2}, we let $F$ be a localization of $R$,
typically its field of fractions, we let $\tr: B \to F$ be the regular
trace, and we assume that the image of $\tr$ is in $R$. (This happens
frequently when $R$ is the center -- see Lemma \ref{yylem2.7}(9), for
example.)

In \cite{CPWZ}, we computed some discriminants.
Here are some new examples.

\begin{example}
\label{yyex1.3}
Let $k$ be a commutative domain such that $2$ is nonzero in $k$.
In parts (2) and (3) we further assume that $3$ is nonzero in $k$ and
that $\xi\in k$ is a primitive third root of unity. Some details
in the computations are omitted.
\begin{enumerate}
\item
Let $R$ be a commutative domain, $0\neq x\in R$, and let
$A=\begin{pmatrix} R& R\\ xR & R\end{pmatrix}$. Then the center of $A$
is $R$ and $Z:=\{e_{11},e_{12}, xe_{21}, e_{22}\}$ is an $R$-basis
of $A$. By using the regular trace $\tr$, we have
\[
\tr(e_{11})=2,\quad
\tr(e_{12})=0,\quad
\tr(xe_{21})=0,\quad
\tr(e_{22})=2.
\]
Using these traces and the fact $\tr$ is $R$-linear, we have the
matrix
\[
(\tr(z_iz_j))_{4\times 4}=\begin{pmatrix} 2 &0 &0 &0 \\
0& 0 &2x&0\\ 0&2x&0&0\\0&0&0&2
\end{pmatrix}
\]
and the discriminant $d(A/R)$ is $-2^4 x^2$.
\item
Let $B=k_{p_{ij}}[x_1,x_2,x_3]$, where $p_{12}=-1$,
$p_{13}=\xi$ and $p_{23}=1$. Then the center $R$
is the polynomial ring generated by $x_1^6$, $x_2^2$ and $x_3^3$.
The algebra $B$ is a free $R$-module with basis
\[
Z:=\{x_1^{i_1}x_2^{i_2}x_3^{i_3}\mid
0\leq i_1\leq 5, 0\leq i_2\leq 1, 0\leq i_3\leq 2\}.
\]
The rank of $B$ over $R$ is $36$. One can check that
the regular traces are
\[
\tr(1)=36, \quad \tr(f)=0 \ \ \forall \ f\in Z\setminus \{1\}.
\]
The discriminant $d(B/R)$ is $(x_1^5 x_2x_3^2)^{36}$
[Proposition \ref{yypro2.8}].
\item
Let $C=k_{p_{ij}}[x_1,x_2,x_3]$, where $p_{12}=-1$,
$p_{13}=-1$, and $p_{23}=1$. Then the center $R$ is generated by
$x_1^2$, $x_2^2$, $x_3^2$, and $x_2x_3$. So $R$ is not a polynomial ring
and $C$ is not free over $R$. The rank of $C$ over $R$ is 4 and
$C$ is generated by the set $\{1, x_1, x_2, x_3, x_1x_2, x_1x_3\}$
over $R$. If $F$ is the field of fractions of $R$, then one can show
that the image of the regular trace $\tr: B \to F$ is in $R$.  By a
degree argument, the regular traces are
\[
\tr(1)=4, \quad \tr(x_1)=\tr(x_2)=\tr(x_3)=\tr(x_1x_2)=\tr(x_1x_3)=0.
\]
Since $C$ is not free over $R$, we compute the modified discriminant
ideal. A non-trivial computation shows that $\MD{4}{C}$ is the ideal
generated by $x_1^4 x_2^i x_3^{4-i}$ for $i=0,1,2,3,4$ and
$d_4(C/R)=_{k^\times} x_1^4$. In this example, it is also possible to
compute $d_v(C/R)$ for other $v$:
\[
d_v(C/R)=_{k^\times}
\begin{cases} 0 & v>4,\\ x_1^2 &v=3,\\ 1 & v<3.\end{cases}
\]
\item
Let $D=k_{p_{ij}}[x_1,x_2,x_3]$, where $p_{12}=-1$,
$p_{13}=-1$, and $p_{23}=i$ where $i^2=-1$. Then the center $R$ is generated by
$x_1^2$, $x_2^4$, $x_3^4$, and $x_1x_2^2x_3^2$. As in the last example,
$R$ is not a polynomial ring and $D$ is not free over $R$, but the
image of the regular trace is in $R$. The rank of $D$
over $R$ is 16 and $D$ is generated by $x_1^i x_2^j x_3^k$,
where $0\leq i\leq 1$,
$0\leq j,k\leq 3$ and $(i,j,k)\neq (1,2,2)$. One can check that
\[
\tr(1)=16, \quad \tr(x_1^ix_2^j x_3^k)=0
\]
for all $0\leq i\leq 1$, $0\leq j,k\leq 3$, and $(i,j,k)\neq (1,2,2)$
[Lemma \ref{yylem2.7}(8)].
The modified discriminant ideal $\MD{16}{D}$ is generated by
$x_2^{48} x_3^{48} \cdot f$, where $f$ ranges over elements of
the form
$(x_1 x_2^{-2} x_3^2)^{i_1} (x_1 x_2^2 x_3^{-2})^{i_2}
(x_1 x_2^{-2} x_3^{-2})^{i_3}$ for all $0\leq i_1, i_2, i_3\leq 8$.
As a consequence, $d_{16}(D/R)=_{k^\times} x_2^{16}x_3^{16}$
[Lemma \ref{yylem1.11}(4)].
\item
Let $E=k_{p_{ij}}[x_1,x_2,x_3]$, where $p_{12}=-1$,
$p_{13}=\xi$, and $p_{23}=-1$. Then the center $R$
is generated by $x_1^6$, $x_2^2$, $x_3^6$, and $x_1^3x_3^3$, which is not
a polynomial ring. The modified discriminant ideal
$\MD{36}{D}$ is generated by
\[
(x_1^2 x_2x_3^2)^{36} x_1^{3i} x_3^{3(36-i)} \quad {\text{for}}
\quad 0\leq i\leq 36,
\]
and so $d_{36}(E/R)=_{k^\times} (x_1^2 x_2x_3^2)^{36}$ [Lemma
\ref{yylem1.11}(4)].
\end{enumerate}
\end{example}

One of key lemmas is the following, which suggests that the
discriminant controls automorphisms.

\begin{lemma}
\label{yylem1.4} Retain the notation as in Definitions
{\rm{\ref{yydef1.1}}}
and {\rm{\ref{yydef1.2}}}. Suppose that $\tr$ is the regular trace and
that the image of $\tr$ is in $R$.
Let $g$ be an automorphism of $B$
such that $g$ and $g^{-1}$ preserve $R$.
\begin{enumerate}
\item \cite[Lemma 1.8(5)]{CPWZ}
The discriminant ideal $D_w(B:\tr)$ is
$g$-invariant, where $w=\rk(B/R)$.
\item \cite[Lemma 1.8(6)]{CPWZ}
If $B$ is a finitely generated free module over $R$, then the
discriminant $d(B/R)$ is $g$-invariant up to a unit of $R$.
\item
The modified discriminant ideal $\MD{v}{B}$ is
$g$-invariant for all $v$.
\item
The $v$-discriminant $d_v(B/R)$ is $g$-invariant up to a unit in $B$,
for all $v$.
\item
For integers $v_1<v_2$, $\MD{v_2}{B} \subset \MD{v_1}{B}$.
So if $d_{v_1}(B/R)$ and $d_{v_2}(B/R)$ exist, then $d_{v_1}(B/R)$
divides $d_{v_2}(B/R)$. As a consequence, the quotient
$d_{v_2}(B/R)/d_{v_1}(B/R)$ is $g$-invariant up to a unit in $B$.
\end{enumerate}
\end{lemma}

\begin{proof} (3) By \cite[Lemma 1.8(2)]{CPWZ}, $\tr(g(x))=
g(\tr(x))$ for all $x\in B$. This implies that $g(d_v(Z,Z':\tr))
=d_v(g(Z),g(Z'):\tr)$ for any $Z,Z'\subset B$. Therefore
$g(\MD{v}{B})\subset \MD{v}{B}$. Similarly,
$g^{-1}(\MD{v}{B})\subset \MD{v}{B}$. These imply that
$g(\MD{v}{B})= \MD{v}{B}$. The proof of (4) is similar.

(5) Let $Z$ and $Z'$ be any $v_2$-element subsets of $B$ as in
Definition \ref{yydef1.2}. Use $X$ for any $v_1$-element subset of
$Z$ and $Y$ for $Z\setminus X$. We similarly define $X'$ and $Y'$. By
linear algebra,
\[
\begin{aligned}
d_{v_2}(Z,Z':\tr)&=\det(\tr(z_iz'_j))_{v_2\times v_2}\\
&=\sum_{X\subset Z,X'\subset Z'}\pm
\det(\tr(x_ix'_j))_{v_1\times v_1}
\det(\tr(y_iy'_j))_{(v_2-v_1)\times (v_2-v_1)}\\
&=\sum_{X\subset Z,X'\subset Z'}\pm
d_{v_1}(X,X':\tr)d_{v_2-v_1}(Y,Y':\tr),
\end{aligned}
\]
which is in $\MD{v_1}{B}$. Hence $\MD{v_2}{B}\subset
\MD{v_1}{B}$ and the second assertion follows. The
consequence is clear.
\end{proof}

The next proposition says that the discriminant controls locally nilpotent
derivations. Recall that a $k$-linear map $\partial: B\to B$ is called a
\emph{derivation} if the Leibniz rule
\[
\partial (xy)=\partial(x)y+x\partial(y)
\]
holds for all $x,y\in B$. We call $\partial$ \emph{locally nilpotent}
if for every $x\in B$, $\partial^n(x)=0$ for some $n$. Given a
locally nilpotent derivation $\partial$ (and assuming that
$\Q \subset k$), the exponential map $\exp(\partial):
B\to B$ is defined by
\[
\exp(\partial)(x)=\sum_{i=0}^{\infty} \frac{1}{i!} \partial^i(x)
\quad {\text{for all $x\in B$.}}
\]
Since $\partial$ is locally nilpotent, $\exp(\partial)$ is an
algebra automorphism of $B$ with inverse $\exp(-\partial)$.

\begin{proposition}
\label{yypro1.5}
Assume that $\Q \subset k$ and that $B^\times =k^\times$. Let $R$ be
the center of $B$. Suppose that $\tr$ is the regular trace and that
the image of $\tr$ is in $R$, and suppose that $d_v(B/R)$ exists. If
$\partial$ is a locally nilpotent derivation of $B$, then
$\partial(d_v(B/R))=0$.  Similarly, if $B$ is finitely generated free
over $R$, then $\partial(d(B/R))=0$.
\end{proposition}

\begin{proof}
For any $c\in k$, consider the algebra automorphism
\[
\exp(c \partial): x\longmapsto
\sum_{i=0}^{\infty} \frac{c^i}{i!} \partial^i(x)
\quad {\text{for all $x\in B$}}.
\]
Let $x=d_v(B/R)$ (or $d(B/R)$ in the second case).
Then, by Lemma \ref{yylem1.4}(4),
$\exp(c\partial) (x)=\lambda_c x\in kx$ for some $\lambda_c\in k^\times$.
This is true for all $c\in \Q$. Since $\partial$ is locally nilpotent,
there are only finitely many nonzero $\partial^i(x)$ terms for
$i=0,1,2,\dots$. By using the
Vandermonde determinant, $\partial^i(x)\in kx$ for all $i$.
If $\partial(x)=ax$, then
$\partial^i(x)=a^i x$ for all $i$. Since
$\partial$ is locally nilpotent, $a=0$ and $\partial(x)=0$.
\end{proof}

This proposition fails when $k$ has positive characteristic
\cite[Example 3.9]{CPWZ}.

Let $C=\bigoplus_i C_i$ be a graded algebra over $k$. We say $C$
is \emph{connected graded} if $C_i=0$ for $i<0$ and $C_0=k$, and
$C$ is \emph{locally finite} if each $C_i$ is finitely generated over
$k$. We now consider filtered rings $A$. Let $Y$ be a finitely
generated free $k$-submodule of $A$ such that $k\cap
Y=\{0\}$. Consider the \emph{standard filtration} defined by $F_n A:=
(k+Y)^n$ for all $n\geq 0$. Assume that this filtration is exhaustive
and that the associated graded ring $\gr A$ is connected graded. For
each element $f\in F_n A\setminus F_{n-1} A$, the associated element
in $\gr A$ is defined to be $\gr f=f+F_{n-1} A\in (\gr_F A)_n$. The
degree of a nonzero element $f\in A$, denoted by $\deg f$, is defined to be
the degree of $\gr f$.


Suppose now $A$ is generated by $Y=\bigoplus_{i=1}^n kx_i$, so with
the standard filtration, the nonzero elements of $Y$ have degree 1. A
monomial $x_1^{b_1}\cdots x_n^{b_n}$ is said to have degree
\emph{component-wise less than} (or, \emph{cwlt}, for short)
$x_1^{a_1}\cdots x_n^{a_n}$ if $b_i\leq a_i$ for all $i$ and
$b_{i_0}<a_{i_0}$ for some $i_0$.  We write $f=cx_1^{b_1}\cdots
x_n^{b_n}+\cwlt$ if $f-cx_1^{b_1}\cdots x_n^{b_n}$ is a linear
combination of monomials with degree component-wise less than
$x_1^{b_1}\cdots x_n^{b_n}$.

\begin{definition}
\label{yydef1.6} Retain the above notation.
Suppose that $Y=\bigoplus_{i=1}^n kx_i$ generates
$A$ as an algebra.
\begin{enumerate}
\item
A nonzero element $f\in A$ is called
\emph{locally $(-s)$-dominating} if, up to a permutation,
$f$ can be written as $f(x_1,x_2,\dots,x_{n-s})$ such that,
for every $g\in \Aut(A)$,  one has
\begin{enumerate}
\item
$\deg f(y_1,\dots,y_{n-s})\geq \deg f$, where
$y_i=g(x_i)$ for all $i\leq n-s$, and
\item
$\deg f(y_1,\dots,y_{n-s})> \deg f$ if, further,
$\deg y_{i_0}>1$ for some $i_0\leq n-s$.
\end{enumerate}
\item 
Suppose $\gr A$ is a connected graded domain.  A nonzero element $f\in
A$ generated by $\{x_1,\dots,x_{n-s}\}$ (up to a permutation of
$\{x_i\}_{i=1}^n$) is called \emph{$(-s)$-dominating} if, for every
$\N$-filtered PI algebra $T$ with $\gr T$ a connected graded
domain, and for every subset $\{y_1,\dots,y_{n-s}\}\subset T$ that is
linearly independent in the quotient $k$-module $T/F_0 T$,
there is a lift of $f$, say $f(x_1,\dots,x_{n-s})$, in the free algebra
$k\langle x_1,\dots,x_{n-s}\rangle$, such that the following hold: either
$f(y_1,\dots,y_{n-s})=0$ or
\begin{enumerate}
\item
$\deg f(y_1,\dots,y_{n-s})\geq \deg f$, and
\item
$\deg f(y_1,\dots,y_{n-s})> \deg f$ if, further,
$\deg y_{i_0}>1$ for some $i_0\leq n-s$.
\end{enumerate}
\end{enumerate}
\end{definition}

If $f=x_1^{b_1}\cdots x_{n-s}^{b_{n-s}}+\cwlt$ for some $b_1,\dots,
b_{n-s}\geq 1$, then $f$ is $(-s)$-dominating: see the proof of
\cite[Lemma 2.2]{CPWZ}. It is easy to check that $(-s)$-dominating
elements are indeed locally $(-s)$-dominating.

Note that the notation of ``$0$-dominating'' is exactly
the notation of ``dominating'' of \cite[Definition 2.1(2)]{CPWZ}
and the notation of ``locally $0$-dominating'' is exactly
the notation of ``locally dominating'' of
\cite[Definition 2.1(1)]{CPWZ}.

\begin{definition}
\label{yydef1.7} Let $(A,Y)$ be defined as above.
In particular, $Y=\bigoplus_{i=1}^n kx_i$ generates
$A$ as an algebra.
\begin{enumerate}
\item
An algebra automorphism $g$ of $A$ is said to be \emph{$(-s)$-affine}
if $\deg g(x_i)=1$ for all but $s$-many values of $i$. A $0$-affine
automorphism is also called an \emph{affine} automorphism
\cite[Definition 2.4(1)]{CPWZ}.
\item
Let $C$ be an algebra over $k$.  A $k$-algebra automorphism $g$ of
$A\otimes C$ is said to be \emph{$(-s)$-$C$-affine} if $g(x_i)\subset
(Y\oplus k)\otimes C$ for all but $s$-many values of $i$. A
$0$-$C$-affine automorphism is also called a \emph{$C$-affine}
automorphism.
\end{enumerate}
\end{definition}

Note that any elementary automorphism is $(-1)$-affine.
The next example shows that not every automorphism is affine.

\begin{example}\label{yyex1.8}
For $q\in k^\times$, let $k_q[x_1,\dots,x_n]$ be the $q$-skew
polynomial ring generated by $\{x_1,\dots,x_n\}$ and subject to the
relations $x_j x_i=q x_i x_j$ for all $i<j$.  Suppose $q$ is a
primitive $\ell$th root of unity for some $\ell>1$.  If $n$ is odd,
then there is an automorphism which is elementary and $(-1)$-affine,
but not affine:
\[
x_i \mapsto
\begin{cases}
x_i & \text{if $i<n$}, \\
x_n+x_1^{\ell-1} x_2 \cdots x_{n-2}^{\ell-1}x_{n-1} & \text{if $i=n$.}
\end{cases}
\]
On the other hand, if $n$ is even, then every automorphism of
$k_q[x_1,\dots,x_n]$ is affine: see the next section.
\end{example}

The Nagata automorphism of the ordinary polynomial algebra
$k[x_1, x_2, x_3]$ is $(-2)$-affine but not $(-1)$-affine \cite{SU}.

The definition of a $(-s)$-affine automorphism (and that of a
$(-s)$-dominating element) depends on $Y$ (or on the filtration of
$A$). But in most cases, there is an obvious choice of filtration.

We conclude this section by proving Theorem~\ref{yythm1.13}. This is a
generalization of the main result of \cite{CPWZ}, namely,
\cite[Theorem 3]{CPWZ}.  We need to develop a few tools, first.
Let $R$ be a central subalgebra of $A$ and
let $F$ be a ring of fractions of $R$ (for example, the field of
fractions of $R$). Write $A_F:=A\otimes_R F$ and suppose that $A_F$ is
finitely generated free over $F$.

Here is a list of linear algebra facts without proof.

\begin{lemma}
\label{yylem1.9} Suppose that $A_F$ is finitely generated free over
$F$ and that $v$ is a positive integer. Let $\tr$ be the regular
trace map $\tr: A_F\to F$. Let $Z:=\{z_i\}_{i=1}^v$ and $Z':=
\{z'_i\}_{i=1}^v$ be subsets of $A$, and suppose $y_1 \in A$.
\begin{enumerate}
\item
Let
$Z_2=\{y_1, z_2,\dots,z_v\}$ and $Z_3=\{y_1+z_1,z_2,\dots,z_v\}.$
Then
\[
d_v(Z_3,Z':\tr)=d_v(Z,Z':\tr)+d_v(Z_2,Z':\tr).
\]
\item
If $Z$ is linearly dependent over $F$, then
$d_v(Z,Z':\tr)=0$.
\item
If $Z_1=\{c z_1,z_2,\dots,z_v\}$ for $c\in F$, then
$d_v(Z_1,Z':\tr)=c d_v(Z,Z':\tr)$.
\item
Let $X$ be a generating set of $A$ over $R$.
Then $d_v(Z,Z':\tr)$ is an $R$-linear combination of
elements $d_v(X_1,X_2:\tr)$, where $X_1$ and $X_2$ consist of $v$
elements in $X$.
\end{enumerate}
\end{lemma}

\begin{definition}\label{yydef1.10}
A subset $b=\{b_1,\dots, b_w\}\subset A$ is called a \emph{semi-basis}
of $A$ if it is an $F$-basis of $A_F$, where $b_i$ is viewed as
$b_i\otimes 1\in A_F$. In this case $w$ is the rank of $A$ over $R$.
The set $b$ is called a \emph{quasi-basis} of $A$ (with respect to
$X$) if
\begin{enumerate}
\item
$b=\{b_1,\dots, b_w\}$  is a semi-basis of $A$, and
\item
There is a set of elements $X=\{x_j\}_{j\in J}$ containing $b$ such
that $A$ is generated by $X$ as an $R$-module and every element
$x_j\in X$ is of the form $c b_i$ for some $c\in F$ and $b_i\in b$.
We denote the element $c$ by $(x_j:b_i)$.
\end{enumerate}
Let $Z:=\{z_1,\dots,z_w\}$ be a subset of $A$. If $b$ is a semi-basis,
then for each $i$,
\[
z_i=\sum_{j=1}^w a_{ij} b_j \quad {\text{for some $a_{ij}\in F$}}.
\]
The $w\times w$-matrix $(a_{ij})$ is denoted by $(Z:b)$.
Let $X$ be a set of generators of $A$ as an $R$-module, and assume
that $X$ contains
$b$. Let $X/b$ denote the subset of $F$ consisting of nonzero scalars
of the form $\det(Z:b)$ for all $Z\subset X$ with
$|Z|=w$. Let
\[
\mathcal{D}(X/b)= \{ d_w(b:\tr) ff' \mid f, f' \in X/b \}.
\]
Note that if $Z$ and $Z'$ are $w$-element subsets of $X$,
then
\begin{align}
\label{1.10.1}
d_w(Z,Z':\tr)&=\det(\tr(z_iz'_j))=\det ((Z:b)(\tr(b_ib_j))(Z':b)^t)\\
&\notag
=\det(Z:b) \det(Z':b)\det
(\tr(b_ib_j))\\
&\notag
=\det(Z:b) \det(Z':b) d_w(b:\tr)\in \mathcal{D}(X/b).
\end{align}
For any integer $v$, define
\[
\mathcal{D}_v(X)=\{d_v(Z,Z':\tr)\mid Z, Z'\subset X\}.
\]
Then $\mathcal{D}_w(X)=\mathcal{D}(X/b)$. As a consequence
of \eqref{1.10.1}, if $Z$ and $Z'$ are two $R$-bases of $A$, then
\begin{equation}
\label{1.10.2}
d_w(Z,Z':\tr)=_{R^\times} d_w(b:\tr).
\end{equation}
If $b=\{b_1, \dots, b_w\}$ is a quasi-basis with respect to
$X=\{x_j\}_{j \in J}$, then for each $i$, let $C_i$ be the set of
nonzero elements of the form $(x_j:b_i)$ for all $j$. It is easy to
see that every element in $X/b$ is of the form $c_1 c_2\cdots c_w$,
where $c_i\in C_i$ for each $i$. Let
\[
\mathcal{D}^c(X/b)=\{ d_w(b:\tr) \prod_{i=1}^w (c_ic'_i)
\mid c_i, c'_i \in C_i\}.
\]
If $b$ is a quasi-basis with respect to $X$, then
$\mathcal{D}(X/b)=\mathcal{D}^c(X/b)$.
\end{definition}

\begin{lemma}
\label{yylem1.11} Let $X$ be a set of generators of $A$ as an $R$-module and
$w=\rk(A/R)$.
\begin{enumerate}
\item
For any $v\geq 1$, the modified $v$-discriminant ideal $\MD{v}{A}$ is
generated by $d_v(Z,Z':\tr)$ for all $Z,Z'\subset X$.
\item
For any $v\geq 1$, the $v$-discriminant $d_v(A/R)$ is the gcd of
$\mathcal{D}_v(X)$.
\item
If $b$ is a semi-basis of $A$, then $d_w(A/R)=\gcd \mathcal{D}(X/b)$.
\item
If $b$ is a quasi-basis of $A$ with respect to $X$, then
$d_w(A/R)=\gcd \mathcal{D}^c(X/b)$.
\end{enumerate}
\end{lemma}

\begin{proof} (1) This follows from Lemma \ref{yylem1.9}(4).

(2), (3) and (4) follow from the definition and part (1).
\end{proof}

Let $C$ be an algebra. We say that $A\otimes C$ is \emph{$A$-closed} if, for
every $0\neq f\in A$ and $x,y\in A\otimes C$, the equation $xy=f$ implies
that $x,y\in A$ up to units of $A\otimes C$. For example,
if $C$ is connected graded and $A\otimes C$ is a domain,
then $A\otimes C$ is $A$-closed.

\begin{lemma}
\label{yylem1.12}
Let $C$ be a $k$-flat commutative algebra such that $A\otimes C$ is a
domain and let $v$ be a positive integer.
\begin{enumerate}
\item
$\xxMD_v(A\otimes C:\tr\otimes C)=\MD{v}{A}\otimes C$.
\item
Suppose $A\otimes C$ is $A$-closed.
If $d_v(A/R)$ exists, then $d_v(A\otimes C/R\otimes C)$ exists and
equals $d_v(A/R)$.
\end{enumerate}
\end{lemma}

\begin{proof}
(1) Let $X$ be a set of generators of $A$ as an $R$-module. Then
$X$ is also a set of generators of $A\otimes C$ as an $R\otimes C$-module.
The assertion follows from Lemma \ref{yylem1.11}(1).

(2) Suppose $d:=d_v(A/R)$ exists. Then it is the gcd of $d_v(Z,Z':\tr)$ in
$A$ for all $Z,Z'\subset X$ [Lemma \ref{yylem1.11}(2)].
Let $d'$ be a common divisor of $d_v(Z,Z':\tr)$ in $A\otimes C$ for
all $Z,Z'\subset X$. Then we may assume that $d'$ is in $A$ by
the $A$-closedness
of $A\otimes C$. Hence $d'$ divides $d$. Therefore $d$ is the gcd of
$\{d_v(Z,Z':\tr)\mid Z,Z'\subset X\}$ in $A\otimes C$.
The assertion follows from Lemma \ref{yylem1.11}(2).
\end{proof}

As before let $A$ be a filtered algebra with standard filtration
$F_n A=(k\oplus Y)^n$, where $Y:=\bigoplus_{i=1}^n kx_i$ generates $A$,
and assume that the associated graded ring $\gr A$ is a connected graded domain.
Let $C(A)$ denote the center of $A$.
The discriminant of $A$ can also control the automorphism group of
$A[t]$. For any $g\in \Aut(A)$, $c\in k^\times$ and $r\in C(A)$, the map
\begin{equation}\label{1.12.1}
 \sigma: t\mapsto ct+r, \quad  x\mapsto g(x) \quad {\text{for all $x\in A$}}
\end{equation}
determines uniquely a \emph{triangular} automorphism of $A[t]$.  The
non-affine automorphisms given in Example \ref{yyex1.8} can be viewed
as elementary triangular automorphisms of the Ore extension
$D[x_n;\tau]$, where $D$ is the subalgebra generated by
$\{x_1,\dots,x_{n-1}\}$. We associate the triangular automorphism
$\sigma$ \eqref{1.12.1} with the upper triangular matrix
$\begin{pmatrix} g & r\\0& c\end{pmatrix}$.  The triangular
automorphisms form a subgroup of $\Aut(A[t])$, denoted by
$\begin{pmatrix} \Aut(A)& C(A)\\0& k^\times\end{pmatrix}$ or
$\Aut_{\tr}(A[t])$.  Explicit examples are computed in \cite[Theorems
4.10 and 4.11]{CPWZ}.

Now we are ready to prove Theorem \ref{yythm0.1},
which is a discriminant criterion for affine automorphisms.

\begin{theorem}
\label{yythm1.13} 
Let $A$ be an algebra and let $Y$ be a $k$-subspace of $A$ which
generates $A$ as an algebra. Give $A$ the standard filtration $F_n A =
(k+Y)^n$ and suppose that the associated graded ring $\gr A$ is a
connected graded domain. Suppose also that $A$ has finite rank over its
center $C(A)$. Assume that there is an integer $v\geq 1$
such that the $v$-discriminant $d_v(A/C(A))$ is locally dominating
with respect to $Y$.  In parts {\rm{(2--5)}} we further assume that
$d_v(A/C(A))$ is dominating with respect to $Y$. Then the following
hold.
\begin{enumerate}
\item
Every automorphism of $A$ is affine.
\item
$\Aut(A[t])=\Aut_{\tr}(A[t]).$
\end{enumerate}
 Suppose  that
$\Z\subset k$ in parts {\rm{(3,4,5)}} and further
that $k$ is a field in
part {\rm{(5)}}.
\begin{enumerate}
\setcounter{enumi}{2}
\item
Every locally nilpotent derivation $\partial$ of $A[t]$ is of the form
\[
\partial(x)=0 \quad {\text{for all $x\in A$}},
\quad \partial(t)=r \quad {\text{for some $r\in R$}}.
\]
\item
Every locally nilpotent derivation of $A$ is zero.
\item
$\Aut(A)$ is an algebraic group that fits into an exact
sequence
\[
1\to (k^\times)^r\to \Aut(A)\to S\to 1
\]
for some finite group $S$. Indeed, $\Aut(A)=S\ltimes
(k^\times)^r$.
 \end{enumerate}
\end{theorem}

\begin{proof} (1) Let $g\in \Aut(A)$.
By Lemma \ref{yylem1.4}(4), $d_v(A/C(A))$ is $g$-invariant.
By \cite[Lemma 2.6]{CPWZ}, $g$ is affine.

(2) Note that $A\otimes k[t]$ is $A$-closed (taking $C=k[t]$),
and $(A[t])^\times=A^\times$.
By Lemma \ref{yylem1.12}(2),
$d_v(A/C(A))=_{A^\times} d_v(A[t]/C(A[t])$. Then the proof of
\cite[Lemma 3.2]{CPWZ} works for $d_v(A/C(A))$. Let $h\in \Aut(A[t])$.
By \cite[Lemma 3.2(2)]{CPWZ}, $h(x_i)\in
Y\oplus k\subset A$, or $h(A)\subset A$.
Applying \cite[Lemma 3.2(3)]{CPWZ}
to $h':=h^{-1}$, we have $h'(A)\subset A$. Thus
$h|_A$ and $h'|_A$ are inverse to each other and
hence $h|_A\in \Aut(A)$. The rest is the same as the proof of
\cite[Theorem 3.5]{CPWZ}.

(3), (4) and (5). By localizing the commutative domain $k$,
we may assume that $k$ is a field of characteristic zero.
The rest of the proof follows from
the proof of \cite[Theorem 3.5(2,3,4)]{CPWZ}.
\end{proof}

In this paper we only consider standard filtrations.
As explained in \cite[Example 5.8]{CPWZ}, the ideas presented
here may be applied to non-standard filtrations.

\section{The discriminant and skew polynomial rings}
\label{yysec2}

In the first half of this section we discuss some properties related
to the center of skew polynomial rings. In the second half of
the section, we compute the discriminant of the skew polynomial ring
over its center.

Recall that the skew polynomial
ring $k_{p_{ij}}[x_1,\dots,x_n]$ is a  connected graded Koszul
algebra that is generated by $x_i$ with $\deg x_i=1$,
and subject to the quadratic relations $x_jx_i=p_{ij}x_ix_j$ for all
$i<j$, where $p_{ij}\in k^\times$ for all $i<j$.
We also write $k_{p_{ij}}[\underline{x}_n]$
for the skew polynomial ring $k_{p_{ij}}[x_1,\dots,x_n]$.
It is well-known that, if $k$ is a field, then
$k_{p_{ij}}[\underline{x}_n]$ is a noetherian domain of
Gelfand-Kirillov dimension,
Krull dimension, and global dimension $n$ \cite{MR}.
If the parameters $p_{ij}$ are generic (and $\ch k=0$),
then $\Aut(k_{p_{ij}}[\underline{x}_n])=(k^\times)^n$
\cite{AlC, Y1}. In this paper we are interested in the case
when the $p_{ij}$ are not generic.

Consider the following two conditions:
\begin{enumerate}
\item[(H1)]
$x_i$ is not central in $k_{p_{ij}}[\underline{x}_n]$ for all
$i=1,\dots, n$.
\item[(H2)]
 $p_{ij}$ is a root of unity for all $i<j$.
\end{enumerate}

Throughout the rest of this section let
$A=k_{p_{ij}}[\underline{x}_n]$. Note that every monomial
$x_1^{d_1}\cdots x_n^{d_n}$ is normal in $A$. Condition (H1) ensures
that $A$ is not a commutative polynomial ring. Condition (H2) implies
that $A$ is PI.
Since $A$ is $\Z^n$-graded with $\deg x_i=(0,\dots, 1,\dots,0)$,
where $1$ is in the $i$th position, the center of $A$ is
$\Z^n$-graded.
Thus the center of $A$ has a $k$-linear basis consisting of
monomials.

\begin{definition}\label{yydef2.1}
For each $i$, define an automorphism $\phi_i$ of $A$, called a
\emph{conjugation automorphism} or \emph{conjugation by $x_i$}, by
\[
\phi_i(x_j)=p_{ij} x_j\quad \forall \; i,j
\]
(where, as earlier, $p_{ii}=1$ for all $i$ and $p_{ij}=p_{ji}^{-1}$ if
$i > j$).
\end{definition}

For each monomial $f:=x_1^{d_1}\cdots x_n^{d_n}$,
$\phi_i(f)=\prod_{j=1}^n p_{ij}^{d_j} f$.  Hence $\phi_i(f)=f$ if and
only if $\prod_{j=1}^n p_{ij}^{d_j}=1$. Since $\phi_i$ is conjugation
by $x_i$, $x_i$ commutes with $f$ if and only if $\phi_i(f)=f$, and
then if and only if $\prod_{j=1}^n p_{ij}^{d_j}=1$.  Define
\[
T=\{(d_1,\dots,d_n)\in \N^n\mid \prod_{j=1}^n p_{ij}^{d_j}=1\;
\forall\;  i\}.
\]
If $W$ is any subset of $\N^n$, let
\[
X^W=\{x_1^{d_1}\cdots x_n^{d_n} \mid (d_1,\dots,d_n)\in W\}.
\]

\begin{lemma}
\label{yylem2.2} Retain the above notation. Then
the following hold.
\begin{enumerate}
\item
The center $C(A)$ of $A$ has a  monomial basis
$\{f\mid f\in X^T\}$.
\item
Assume {\rm{(H2)}} and that $k$ is a field. Then $C(A)$ is
Cohen-Macaulay.
\end{enumerate}
\end{lemma}

\begin{proof} (1) This is clear.

(2) If $\ch k=0$, this is well-known \cite[Theorem 2.2(3)]{SVdB}.
Now we assume that $\ch k=p>0$.
Let $S$ be the abelian group generated by the conjugation
automorphisms $\phi_i$.  Then $S$ is a finite
group and $C(A)$ is the fixed subring $A^S$. The order of
$\phi_i$ equals the order of the subgroup $G$ of $k^{\times}$ generated by
$\{p_{i1},p_{i2},\dots,p_{in}\}$. Since $G$ is a subgroup of $k^\times$,
it is cyclic. We may assume that the base field
$k$ is finite. Then $|k|=p^N$ for some $N$
and $k^\times$ is a cyclic group of order $p^N-1$. Thus the order of $G$
is coprime to $p$. Since each $\phi_i$ has order coprime to $p$,
the order of $S$ is coprime to $p$. As a consequence, the group
algebra $k S$ is semisimple.
Then $A^S$ is Cohen-Macaulay by \cite[Lemma 3.2(b)]{KKZ2} (note
that the  proof of \cite[Lemma 3.2(b)]{KKZ2} only uses the fact
$k S$ is semisimple, not the hypothesis $\ch k=0$).
\end{proof}

When $n$ is large, it is not easy to understand $C(A)$ or $T$ completely.
The following lemma is useful in a special case.

\begin{lemma}
\label{yylem2.3} Assume {\rm{(H2)}}. The following are equivalent.
\begin{enumerate}
\item
The center $C(A)$ is a polynomial ring.
\item
There are positive integers $a_1,\dots,a_n$ such that
$(d_1,\dots,d_n)\in T$ if and only if $a_i\mid d_i$ for all $i$.
In other words, $T$ is generated by $(0,\dots, 0,a_i, 0,\dots, 0)$
for $i=1,\dots,n$, where $a_i$ is in the $i$th position,
and the $\N^n$-solutions $(d_1,\dots, d_n)$
to the system of equations
\begin{equation}\label{2.2.1}
\prod_{j=1}^n p_{ij}^{d_j}=1, \quad {\text{for all $i=1,\dots,n$}}
\end{equation}
form the set $\{\sum_{i=1}^n b_i (0,\dots, 0,a_i, 0,\dots, 0)
\mid b_i\geq 0\}$.
\item
There are positive integers $a_1,\dots,a_n$ such that
$C(A)$ is generated by $x_i^{a_i}$ for $i=1,\dots, n$.
\item
$A$ is finitely generated free over $C(A)$.
\end{enumerate}
\end{lemma}

\begin{proof} (1) $\Rightarrow$ (2) By localizing $k$,
we may assume that $k$ is a field.
Since $C(A)$ and $A$ have the same Gelfand-Kirillov dimension,
the number of generators in $C(A)$ must be $n$.
Let $a_i$ be the minimal integer such that $(0,\dots, 0,a_i, 0,\dots, 0)
\in T$. Then $x_i^{a_i}\in C(A)$, but is not generated by any other elements
in $C(A)$. Let $\mathfrak{m}$ be the graded ideal $C(A)_{\geq 1}$.
Then the images of $x_i^{a_i}$ in $\mathfrak{m}/\mathfrak{m}^2$
(still denoted by $x_i^{a_i}$)
are linearly independent elements; since $\mathfrak{m}/\mathfrak{m}^2$
is a free $k$-module of rank $n$, we have
$\mathfrak{m}/\mathfrak{m}^2=\oplus_{i=1}^n k x_i^{a_i}$.
Thus $C(A)$ is generated by $x_i^{a_i}$. The assertion follows.

(2) $\Rightarrow$ (3) and
(3) $\Rightarrow$ (4) are clear.

(4) $\Rightarrow$ (1) Let $F$ be the field of fractions of $k$.
Then $A\otimes F$ is finitely generated free over $C(A)\otimes F$.
Since $F$ is a field, $C(A)\otimes F$ has global dimension $n$
\cite[Lemma 1.11]{KKZ1}. The only connected graded commutative
algebra of finite global dimension
is the  polynomial ring. So $C(A)\otimes F$ is a
polynomial ring.  By the proof of
(1) $\Rightarrow$ (2) for $k=F$, $C(A)\otimes F$ is generated by
$x_i^{a_i}$s. Therefore $C(A)$ is generated by $x_i^{a_i}$.
Thus $C(A)$ is a polynomial ring.
\end{proof}

\begin{example}
\label{yyex2.4} Let $q$ be a primitive $\ell$th root of unity,
and write $p_{ij}=q^{\phi_{ij}}$ for some integers $\phi_{ij}$.
\begin{enumerate}
\item
 If  $\det(\phi_{ij})$ is invertible in $\Z/(\ell)$,
then the center
of $A$ is $k[x_1^{\ell},\dots,x_n^{\ell}]$. To see this, let
$x_1^{d_1}\cdots x_n^{d_n}$ be in the center. By the
definition of $T$, we have $\prod_{j=1}^n p_{ij}^{d_j}=1$
for all $i$, or equivalently,
\[
\sum_{j=1}^n \phi_{ij} d_j\equiv 0 \pmod{\ell}.
\]
Since $\det(\phi_{ij})_{n\times n}$ is invertible in $\Z/(\ell)$,
$d_i \equiv 0\pmod{\ell}$, or $\ell \mid d_i$. It is clear that
$x_i^{\ell}\in C(A)$. Thus $C(A)=k[x_1^{\ell},\dots,x_n^{\ell}]$.

Note that we may take $\phi_{ii}=0$
and $\phi_{ji}=-\phi_{ij}$. Then the matrix  $(\phi_{ij})$ is
skew-symmetric. Hence $\det(\phi_{ij})$ being
invertible can only  happen when $n$ is even.
\item
A special case of (1) is when $\phi_{ij}=1$ for all $i<j$ (or
$p_{ij}=q$ for all $i<j$). When $n$ is even, then, by linear
algebra, $\det(\phi_{ij})=1$, which is invertible for any $\ell$.
In this case the center of $k_q[x_1,\dots,x_n]$ is
$k[x_1^{\ell},\dots,x_n^{\ell}]$.
\item
When $n$ is odd, there are different kinds of examples
for which $C(A)$ is a polynomial ring.
Let $n=3$ and $q$ be a primitive $\ell$th root of unity. Suppose $\ell=abc$,
where $a,b,c\geq 2$ are pairwise coprime. Let $p_{12}=q^{ab}$,
$p_{13}=q^{ac}$, and $p_{23}=q^{bc}$. Then one can check that the center of
$k_{p_{ij}}[x_1,x_2,x_3]$ is $k[x_1^{bc},x_2^{ac},x_3^{ab}]$.
Higher dimensional examples can be constructed in a similar way.
\item
Again let $n=3$, $q$ be a primitive $\ell$th root of unity,
and $\ell=abc$, where $a,b,c\geq 2$ are pairwise coprime.
Let $p_{12}=q^a$, $p_{13}=q^{-b}$, and $p_{23}=q^c$.
Then the center $C(A)$ is not a polynomial ring. To see this, note that the
monomials $x_1^{\ell}, x_2^{\ell},x_3^{\ell},
x_1^cx_2^bx_3^a$, and so on, are generators of $C(A)$,
but $x_1^c$ is not in the  center. By Lemma \ref{yylem2.3},
$C(A)$ is not a polynomial ring, and in this case,
$C(A)\subset k\langle  x_1^c, x_2^b, x_3^a\rangle$.
\end{enumerate}
\end{example}

Note that under the hypothesis (H2), the subgroup of
$k^\times$ generated by $\{p_{ij}\}$ is $\langle q\rangle$
for some  root of unity $q$.

\begin{lemma}
\label{yylem2.5}
Assume {\rm{(H1)}} and {\rm{(H2)}}.
Assume  that the group generated by $\{p_{ij}\}$
is $\langle q\rangle$, where $q$ is a primitive $\ell$th root of
unity and $\ell$ is a prime number. If $C(A)$ is not a polynomial ring,
then there is a solution $(d_1,d_2,\dots,d_n)\in \N^n$ to
the system of equations
\[
\prod_{j=1}^n p_{ij}^{d_j}=1, \quad {\text{for all $i=1,\dots,n$}}
\]
such that $d_s=1$ for some $s$.
\end{lemma}

\begin{proof} Since $x_i \not \in C(A)$ and $x_i^\ell \in C(A)$, we
have
\[
\ell = \min_{a > 0} \{x_i^a \in C(A) \}.
\]
Since $C(A)$ is not a polynomial ring, there is a solution
$d:=(d_1,d_2,\dots,d_n)$ to system of equations given in the lemma
such that some $d_s$ is not divisible by $\ell$.  Note that any
multiple of $d$ is still a solution. By replacing $d$ by a multiple of
$d$, we have $d_s\equiv 1\pmod{\ell}$ (as $\ell$ is prime).  Finally,
by replacing $d_s$ by $1$ (as $p_{is}^\ell=1$) we obtain the desired
solution.
\end{proof}

Next we compute the discriminant $d(A/R)$ when $R$ is a polynomial ring. We
start with an easy lemma. Let $\Lambda$ be
an abelian group and let $B$ be a $\Lambda$-graded algebra.
Then the center of $B$ is also $\Lambda$-graded.

\begin{lemma}
 \label{yylem2.6} Let $B$ be a
 $\Lambda$-graded algebra and $R$ a central graded subalgebra of $B$.
Suppose that
 $R^\times =k^\times$. For every $v\geq 1$ and any sets of homogeneous
elements $Z=\{z_i\}_{i=1}^v$ and $Z'=\{z'_i\}_{i=1}^v$, the
discriminant $d_v(Z,Z':\tr)$ is either $0$ or homogeneous of degree
$\sum_{i=1}^v (\deg z_i+\deg z'_i)$. As a consequence, if
$B$ is a finitely generated graded free module over
$R$, then $d(B/R)$ is homogeneous.
\end{lemma}

\begin{proof} The consequence is clear, so we prove the main
assertion.

Let $F$ be the graded field of fractions of $C(B)$.  Since $C(B)$ is
graded, we can choose a semi-basis $b=\{b_1,\dots,b_w\}$ of $B$
consisting of homogeneous elements $b_i$, where $w=\rk(B/C(B))$.  Then
$B$ is a finitely generated graded free module over $F$ with basis
$b$. For each homogeneous element $f$, $\tr(f)$ is either 0 or
homogeneous of degree $\deg (f)$.  In particular, $\tr(z_iz'_j)$ is
either 0 or homogeneous of degree
$\deg(z_iz'_j)=\deg(z_i)+\deg(z'_j)$. By definition, $d_v(Z,Z':\tr)$ is
the determinant $\det (\tr(z_iz'_j))_{v\times v}$, which is a signed
sum of elements
\[
\sum_{\sigma\in S_v} \tr(z_1 z'_{\sigma(1)})\tr(z_2 z'_{\sigma(2)})
\cdots \tr(z_v z'_{\sigma(v)}).
\]
Each above element is either 0 or homogeneous of degree
$\sum_{i=1}^v (\deg z_i+\deg z'_i)$. Hence
$d_v(Z,Z':\tr)$ is either 0 or homogeneous of degree
$\sum_{i=1}^v (\deg z_i+\deg z'_i)$.
\end{proof}

We may consider $k_{p_{ij}}[\underline{x}_n]$
as either $\Z$-graded or $\Z^n$-graded.
Let $k_{p_{ij}}[\underline{x}_n^{\pm 1}]$ denote the
algebra $k_{p_{ij}}[x_1^{\pm 1},\dots, x_n^{\pm 1}]$.
By a monomial in $k_{p_{ij}}[\underline{x}_n^{\pm 1}]$, we mean an element of
the form $c x_1^{a_1}\cdots x_n^{a_n}$ for some $a_i\in \Z$
and some $0\neq c\in k$.

\begin{lemma}
\label{yylem2.7} Let $A=k_{p_{ij}}[\underline{x}_n]$ and
$B=k_{p_{ij}}[\underline{x}_n^{\pm 1}]$ with the natural
$\Z^n$-grading. Let $C(A)$ be the center of $A$.
In parts {\rm{(6)--(9)}} suppose {\rm{(H2)}} and let
$\tr: A\to F$ be the regular trace, where $F$ is the field of
fractions of $C(A)$.
\begin{enumerate}
\item
Every homogeneous element in $B$ is a monomial.
\item
Let $f$ be a homogeneous element in $B$.
Then $f\in A$ if and only if $\deg f\in \N^n$.
\item
For any set $\mathcal{D}$ of monomials in $A$, $\gcd \mathcal{D}$
exists and is a monomial.
\item
The center of $A$ {\rm{(}}respectively, $B${\rm{)}} is a
$\Z^n$-graded
subalgebra of $A$ {\rm{(}}respectively, $B${\rm{)}}.
\item
There is a generating set $X$ of the $C(A)$-module $A$ consisting
of monomials.
\item
If $A$ satisfies {\rm{(H2)}}, then $A$ has a quasi-basis $b$.
\item
The rank $\rk(A/C(A))$ is nonzero in $k$.
\item
For every monomial $f\in B$, $\tr(f)\neq 0$ if and only if
$f\in C(B)$.
\item
The image of $\tr: A\to F$ is in $C(A)$.
\end{enumerate}
\end{lemma}

\begin{proof} (1)--(5) are straightforward.

(6) Since $B$ is a graded division ring,
its center is a graded field. Hence $B$ is finitely generated
graded free over $C(B)$ with a basis $b\subset X$. It is easy
to check that $b$ is a quasi-basis.

(7) Since $\rk(A/C(A))=\rk(B/C(B))$, it suffices to show that
$\rk(B/C(B))$ is nonzero in $k$. By localizing $k$, we may
assume that $k$ is a field. If $\ch k=0$, the assertion is trivial,
so we assume that $\ch k=p>0$.

Following the proof of Lemma \ref{yylem2.2}(2), let $S$ be the
abelian group generated by the automorphisms $\phi_i$.  Then $S$ is a
finite group and $C(B)=B^S$. Since each $p_{ij}$ is a root of unity,
by replacing $k$ by the subfield generated by the $p_{ij}$'s, we may
assume $k$ is finite. Then $|k|=p^N$ for some $N$. By the proof of
Lemma \ref{yylem2.2}(2), the order of $S$ is coprime to $p$.
Since $C(B)$ is a $\Z^n$-graded field, $B$ is a finite
dimensional free module over $C(B)$ with a monomial basis
$b=\{b_1=1,\dots,b_w\}$. Let $S^\vee$ be the dual group of $S$.
Define a map $\Phi: b\to S^\vee$ by $\Phi(b_i)(\phi_j)=\phi_j(b_i)
b_i^{-1}$. If $b_ib_j=b_k c$ for some $c\in C(B)$, then one
can check that $\Phi(b_i)\Phi(b_j)=\Phi(b_k)$. This observation
implies that $\Phi$ is injective and the image of $\Phi$ is a
subgroup of $S^\vee$. Therefore the order of $b$, namely, $\rk(B/C(B))$,
is a divisor of $|S|$, which is coprime to $p$. Equivalently,
$\rk(B/C(B))\neq 0$ in $k$.

(8) The regular trace map $\tr: A\to F$ (or $\tr: B\to F$) can be
defined by composing
\[
\tr: A\to A\otimes_{C(A)} C(B)=B\xrightarrow{lm} M_w(C(B))
\xrightarrow{\tr_{\INT}} C(B)\xrightarrow{=} F,
\]
where $lm$ is the left multiplication map. For any monomial
$f$ in $A$ (or in $B$), $tr(f)$ is either zero or of
degree equal to $\deg(f)$ -- that is, the map $\tr$ is homogeneous of
degree 0 with respect to the $\Z^n$-grading. Thus if $\tr(f) \neq
0$, then $\tr(f) \in C(B)$ is a scalar multiple of $f \in B$, so $f$
is in $C(B)$. If $f\in C(B)$, then $\tr(f)=w f$, where $w=\rk(A/C(A))$
is nonzero in $k$, by part (7).

(9) Since the map $\tr$ is homogeneous of degree 0,
the image $\im \tr (A)$
is in $A$ by part (2). Hence $\im \tr(A) \subset A\cap C(B)=C(A)$.
\end{proof}

\begin{proposition}
\label{yypro2.8} Consider $A$ as a $\Z^n$-graded algebra.
Let $R=k[x_1^{\alpha_1},\dots,x_n^{\alpha_n}]$ be a central
subalgebra of $A$, where the $\alpha_i$ are positive integers.
Let $r=\prod_{i=1}^n \alpha_i$. Then
\[
 d(A/R)=_{k^\times} r^r
(\prod_{i=1}^n x_i^{\alpha_i-1})^r.
\]
As a consequence, if $R$ is the center of
$A$ and $\alpha_i>1$ for all $i$, then
$d(A/R)$ is dominating.
\end{proposition}

\begin{proof}
First note that there is a graded basis $Z:=\{x_1^{\beta_1}\cdots x_n^{\beta_n}
\mid 0\leq \beta _i <\alpha_i \; \forall i\}$ of $A$ over $R$, so the
rank of $A$ over $R$ is
$r=\prod_{i=1}^n \alpha_i$.

Let $b:=\{z_1=1,z_2, \dots, z_r\}$ be a monomial basis of $A$ over
$R$. For every element $z_j:=x_1^{\beta_1}\cdots x_n^{\beta_n}$ in
the basis $b$, let $z'_j$ be the monomial $x_1^{\beta'_1}\cdots
x_n^{\beta'_n}$, where
\[
\beta'_i=\begin{cases} 0& \text{if $\beta_i=0,$}
\\ \alpha_i-\beta_i& \text{if $\beta_i\neq 0$.}
\end{cases}
\]
One can check that $z'_j$ is the unique element in the basis such that
$z_j z_j'\in R$. For example, $z'_1=1$.
Then $\tr(z_j z_s)=0$ unless $z_s=z_j'$, and in that case
$\tr(z_j z_j')=r z_jz_j'$. Therefore
$\det (\tr(z_iz_j))=_{k^\times} r^r \prod_{j=1}^r
z_jz_j'$. An easy combinatorial argument gives the result.

For the consequence, note that the rank $r$ is nonzero in $k$ by Lemma
\ref{yylem2.7}(7).  Then $d(A/R)$ is of the form given in \cite[Lemma
2.2(1)]{CPWZ}, which is dominating.
\end{proof}

From Lemma \ref{yylem2.3}, we see that if the center of $A$ is a
commutative polynomial ring, then the center is of the form
$k[x_1^{\alpha_1},\dots, x_n^{\alpha_n}]$. So as an immediate
consequence of this, together with Proposition \ref{yypro2.8} and
Theorem \ref{yythm0.1}, if (H2) holds and if the center of $A$ is a
commutative polynomial ring, then every automorphism of $A$ is affine
and every automorphism of $A[t]$ is triangular.

We also consider the discriminant when $C(A)$
is not a polynomial ring. The goal is an explicit
condition that ensures that the discriminant is dominating.
We recall some notation.
Fix a parameter set $\{p_{ij}\mid 1\leq i<j\leq n\}$ and impose the
usual conditions ($p_{ji}=p_{ij}^{-1}$, $p_{ii}=1$) to define $p_{ij}$
for all $1 \leq i,j \leq n$. For any $1\leq s\leq n$, let
\[
T_s=\{(d_1,\dots, \widehat{d}_s,\dots, d_n)\in \N^{n-1}
\mid \prod_{\substack{j=1\\j\neq s}}^n p_{ij}^{d_j}=p_{is} \; \forall \;
i\neq s\}.
\]

\begin{lemma}
\label{yylem2.9} Retain the above notation.
\begin{enumerate}
\item
If $(d_1,\dots,\widehat{d}_s,\dots,d_n)\in T_s$, then the
equation
$\prod_{j=1,j\neq s}^n p_{ij}^{d_j}=p_{is}$ also holds
for $i=s$.
\item
$T_s=\{(d_1,\dots, \widehat{d}_s,\dots, d_n)\in \N^{n-1}
\mid x_1^{d_1}\cdots x_s^{-1} \cdots x_n^{d_n} \in C(k_{p_{ij}}
[\underline{x}_n^{\pm 1}])\}.$
\end{enumerate}
\end{lemma}

\begin{proof} Both are easy to check.
\end{proof}

By Lemma \ref{yylem2.9}(1), $(d_1,\dots, \widehat{d}_s,\dots, d_n)\in T_s$
if and only if $(d_1,\dots, \widehat{d}_s,\dots, d_n)$ is an
$\N^{n-1}$-solution
to the system of equations
\begin{equation}
\label{2.9.1}
\prod_{\substack{j=1\\j\neq s}}^n p_{ij}^{d_j}=p_{is}, \; \forall \; i.
\end{equation}

The next lemma is easy and the proof is omitted.

\begin{lemma}
\label{yylem2.10} Let $B$ be a $\Lambda$-graded domain, where
$\Lambda$ is a linearly ordered group. Let $c$ be a homogeneous
element in $B$ and $a,b\in B$ such that $ab=c$. Then both $a$
and $b$ are homogeneous.
\end{lemma}

For
$d:=(d_1,\dots,\widehat{d}_s,\dots,d_n)$ in $T_s$,
define $f_d=x_1^{d_1}\cdots \widehat{x}_s\cdots x_n^{d_n}$.
Then by Lemma \ref{yylem2.9}(2), $x_i f_d=p_{si} f_d x_i$
for all $i$. Therefore the map
\begin{equation}\label{2.10.1}
g(cf_d,s): x_i\mapsto \begin{cases}
             x_i &{\text{if $i\neq s$,}}\\
             x_s+cf_d &{\text{if $i=s$}}
            \end{cases}
\end{equation}
extends to an algebra automorphism of $A$, where $c\in k$.
The map
\begin{equation}
\label{2.10.2}
\partial(cf_d, s): x_i\mapsto \begin{cases}
             0 &\qquad\quad{\text{if $i\neq s$,}}\\
             cf_d &\qquad\quad{\text{if $i=s$}}
            \end{cases}
\end{equation}
extends to a locally nilpotent derivation of $A$.
By slight abuse of notation, we let
\[
X^{T_s}=\{x_1^{d_1}\cdots \widehat{x}_s \cdots x_n^{d_n} \mid
(d_1,\dots,\widehat{d}_s,\dots ,d_n)\in T_s\}.
\]
If $F$ is a linear combination of monomials in $X^{T_s}$, we can
define $g(F,s)$ and $\partial(F,s)$ similarly. Automorphisms of the
form $g(F, s)$ are called \emph{elementary} automorphisms. It is easy
to check that $g(F,s) g(F',s)=g(F+F',s)$ as long as both $F$ and $F'$
are linear combinations of monomials in $X^{T_s}$.  As a consequence,
$g(F,s)^{-1}=g(-F,s)$.

\begin{theorem}
\label{yythm2.11}
Let $A=k_{p_{ij}}[\underline{x}_n]$ be a skew polynomial ring
satisfying {\rm{(H2)}}. Let $w=\rk(A/C(A))$.
\begin{enumerate}
\item
For any positive integer $v$, the $v$-discriminant $d_v(A/C(A))$
exists. Furthermore, $d_w(A/C(A))$ is nonzero.
\item
For any $1\leq s\leq n$,
$T_s= \emptyset$ if and only if $x_s\mid d_w(A/C(A))$.
\item
$T_i= \emptyset$ for all $i=1,\dots,n$ if and only if
$d_w(A/(C(A))$ is dominating.
\end{enumerate}
\end{theorem}

\begin{proof} (1) By Lemma \ref{yylem2.7}(5), there is a generating
set $X$ of $A$ over $C(A)$ consisting of monomials. For any
$v$-element subsets $Z,Z'\subset X$, $d_v(Z,Z':\tr)$ is homogeneous by
Lemma \ref{yylem2.6}, and is a monomial in $A$ by Lemma
\ref{yylem2.7}(2).  Applying Lemma \ref{yylem2.7}(3) to the set of
monomials of the form $d_v(Z,Z':\tr)$ for all such $Z$ and $Z'$, we
see that $d_v(A/C(A))$ exists.

For the second assertion it suffices to show that there are
$Z, Z'$ such that
\[
d_w(Z,Z':\tr)\neq 0,
\]
as $d_w(A/C(A))$
is the gcd of such elements. Let $Z=\{z_1=1,z_2,\dots,z_w\}$ be a
quasi-basis of $A$. For each $i$, define $z'_i\in A$ to be a nonzero
monomial such that $z_iz'_i\in C(A)$, and let $Z'=\{z'_i\}_{i=1}^w$.
Then $z_iz'_j\not\in C(A)$ for all $i\neq j$, whence by Lemma
\ref{yylem2.7}(8),
\[
\tr(z_iz'_j)=\begin{cases} w z_iz'_i & i=j,\\ 0& i\neq j.
\end{cases}
\]
Hence $d_w(Z,Z':\tr)=w^w \prod_{i=1}^w (z_iz'_i)$ which is a nonzero
monomial as $w\neq 0$ [Lemma \ref{yylem2.7}(7)].

(2)  By using Lemma \ref{yylem2.9}(2), if $T_s$ is empty, then
$x_1^{d_1}\cdots x_s^{-1} \cdots x_n^{d_n}$ is not in $C(B)$ for
any $d_i\in \N$ for all $i\neq s$ (with $B$ as defined in Lemma
\ref{yylem2.7}).

Let $b=\{b_1, \dots, b_w \}$ be a quasi-basis with respect to a
generating set $X$ [Lemma \ref{yylem2.7}(6)]; we may assume that $X$ contains $x_s$.
Let $Z=\{z_1,\dots,z_w\}$ be a subset of $X$.
We claim that $x_s$ divides $d_w(Z,Z':\tr)$ for all $Z'$.
If $d_w(Z,Z':\tr)=0$, then the claim follows.  If $\det(Z:b)=0$,
then $d_w(Z,Z':\tr)=0$,
so we assume that $\det(Z:b)\neq 0$. Since $\deg d_w(Z,Z':\tr)=
\deg (\prod_{i=1}^w z_i z'_i)$ [Lemma \ref{yylem2.6}],
it's enough to show that $x_s$ divides $z_i$ for some $i$.
Since $b$ is a quasi-basis,
up to a permutation, for each $i$, $z_i=b_i c_i$ for some $0\neq
c_i\in C(B)$. Hence
$Z$ is a quasi-basis of $A$. Therefore, there is an $i$ such that
$x_s=z_i c$ for some
$c\in C(B)$, or $z_i=x_s c^{-1}$. Since the $x_s$-degree of $c$
can not be $1$,
the $x_s$-degree of $x_s c^{-1}$ is not zero. This means that
$x_s$-degree of $z_i$
is nonzero, or $x_s\mid z_i$.

If $T_s$ is non-empty, pick an element  in $T_s$ of the form
\[
d'=(d_1'+m\ell, d_2',\dots,
\widehat{d}_s',\dots,d_n')
\]
with $m\gg 0$. Hence there is
a monomial $f_{d'}$ in $X^{T_s}$ with degree larger than the
degree of $d:=d_w(A/A(C))$. Let $g=g(f_{d'},s)$ be the automorphism
constructed in \eqref{2.10.1}. Then $\deg g(x_s)>\deg d$.
It follows from Lemmas \ref{yylem2.6}
and \ref{yylem2.7} that $d$ is homogeneous,
whence it is a nonzero monomial, say $cx_1^{a_1}\cdots x_n^{a_n}$.
Then we have
\[
\deg d=\deg g(d)=\deg (g(x_1))^{a_1}\cdots
(g(x_n))^{a_n}=\sum_i a_i \deg g(x_i).
\]
If $a_s>0$, then
\[
\deg g(x_s)\leq a_s\deg g(x_s)\leq \sum_i a_i \deg g(x_i)=\deg d,
\]
which contradicts the fact $\deg g(x_s)>\deg d$. Therefore $a_s=0$ and
$x_s$ does not divide $d$.

(3)  Since $d_w(A/C(A))$ is a monomial,
it is of the form $x_1^{a_1}\cdots x_n^{a_n}$, up to a scalar.
The assertion follows from part (2).
\end{proof}

\begin{corollary}
\label{yycor2.12}
Let $A=k_q[\underline{x}_n]$ be a $q$-skew polynomial ring and $q$
a primitive $\ell$th root of unity for some $\ell\geq 2$. Let $w$ be the rank of
$A$ over its center. Then
\[
d_w(A/C(A))=\begin{cases} c \prod_{i=1}^n x_i^{\ell^n(\ell-1)} & {\text{if $n$
is even}}\\c  & {\text{if $n$ is odd}},\end{cases}
\]
for some $0\neq c\in k$.
As a consequence, $\Aut(A)$ is affine if and only if $n$ is even.
\end{corollary}

\begin{proof} First we assume that $n$ is even. By Example \ref{yyex2.4}(2),
the center of $A$ is $k[x_1^{\ell},\dots,x_n^{\ell}]$. Then the
discriminant is given by Proposition \ref{yypro2.8}. By Theorem \ref{yythm1.13},
$\Aut(A)$ is affine. An easy computation gives the formula \eqref{0.3.1}.

If $n$ is odd, then $(\ell-1,1,\ell-1,\dots,\widehat{d}_s,\dots,1, \ell-1)\in T_s$
when $s$ is odd and $(1,\ell-1,\dots,\widehat{d}_s,\dots,\ell-1,1)\in T_s$
when $s$ is even. By Theorem \ref{yythm2.11}(2), $d_w(A/C(A))$ is a
constant. By construction \eqref{2.10.1}, $\Aut(A)$ is not affine.
\end{proof}

\section{Affine and tame automorphisms of skew polynomial rings}
\label{yysec3}

In this section we reprove and extend some results
of Alev and Chamarie about the automorphism groups of skew polynomial
rings \cite{AlC}.  Here is one of the main results in this section.
Let $\LNDer(B)$ denote the set of all locally nilpotent derivations of
an algebra $B$. As in the previous section, let $A$ be
$k_{p_{ij}}[\underline{x}_n]$.

\begin{theorem}
\label{yythm3.1}
Let $A=k_{p_{ij}}[\underline{x}_n]$ be a skew polynomial ring
satisfying {\rm{(H2)}}. The following are equivalent.
\begin{enumerate}
\item
$\Aut(A)$ is affine.
\item
$C(k_{p_{ij}}[\underline{x}_n^{\pm 1}])
\subset k\langle x_1^{\pm \alpha_1},\dots, x_n^{\pm \alpha_n}
\rangle$
for some $\alpha_1,\dots, \alpha_n\geq 2$.
\item
$C(k_{p_{ij}}[\underline{x}_n])
\subset k\langle x_1^{\alpha_1},\dots, x_n^{\alpha_n}
\rangle$
for some $\alpha_1,\dots, \alpha_n\geq 2$.
\item
$T_s=\emptyset$ for all $s=1,\dots,n$.
\item
$d_w(A/C(A))$ is dominating where $w=\rk(A/C(A))$.
\item
$d_w(A/C(A))$ is locally dominating where $w=\rk(A/C(A))$.
\end{enumerate}
If $\Z\subset k$, then the above are also equivalent to
\begin{enumerate}
\item[(7)]
$\LNDer(A)=\{0\}$.
\end{enumerate}
\end{theorem}

The proof of Theorem \ref{yythm3.1} is given in the middle of
the section. One immediate question is, for what kind of
noetherian connected graded Koszul PI algebras is some version
of Theorem \ref{yythm3.1} still valid?

Let $B$ be a connected $\N$-graded algebra generated
in degree 1. Let
$\Aut_{\gr}(B)$ be the subgroup of graded automorphisms of $B$.
An automorphism $g$ of $B$ is called \emph{unipotent} if
$g(v)=v+{\text{(higher degree terms)}}$
for all $v\in B_1$. Let $\Aut_{\uni}(B)$ denote the subgroup of
$\Aut(B)$ consisting of unipotent automorphisms.

In what follows, we do not assume (H2) unless explicitly stated.

\begin{lemma}
\label{yylem3.2} The following are equivalent for $A$.
\begin{enumerate}
\item
$A$ satisfies {\rm{(H1)}}, namely, $x_i$ is not central for all $i$.
\item
For each $i$, there is a $j$ such that $p_{ij}\neq 1$.
\item
For every commutative domain $C\supseteq k$ and for every
$k$-algebra automorphism $g$ of $A\otimes C$,
the constant term of $g(x_i)$ is zero.
\item
$\Aut(A)=\Aut_{\gr}(A)\ltimes \Aut_{\uni}(A)$.
\item
For every commutative domain $C\supseteq k$ and for every
$k$-algebra derivation $\partial$ of $A\otimes C$,  the constant
term of $\partial(x_i)$ is zero.
\item
For every commutative domain $C\supseteq k$ and for every
$k$-algebra locally nilpotent
derivation $\partial$ of $A\otimes C$,  the constant
term of $\partial(x_i)$ is zero.
\end{enumerate}
\end{lemma}

\begin{proof} It is clear that (1) $\Leftrightarrow$ (2) and that
(5) $\Rightarrow$ (6).

(3) $\Rightarrow$ (1) If $x_i$ is central, then
$g: x_j\to x_j+ \delta_{ij}$ defines an algebra automorphism
for which the constant term of $g(x_i)$ is not zero.

(2) $\Rightarrow$ (3) Suppose $g\in \Aut(A\otimes C)$ such that
$g(x_i)=c_i+ y_i$, where $c_i\in C$ is the constant term of $g(x_i)$.
Suppose $c_i\neq 0$ for some $i$. Pick $j$ such that
$p_{ij}\neq 1$. Applying $g$ to the equation
$x_j x_i=p_{ij}x_i x_j$ we have
\[
(c_j+y_j)(c_i+y_i)=p_{ij}(c_i+y_i)(c_j+y_j).
\]
By comparing constant terms, we have
$c_jc_i=p_{ij}c_ic_j$. Since $p_{ij}\neq 1$ and $c_i\neq 0$,
we have $c_j=0$ (as $C$ is a domain), and
\[
y_j(c_i+y_i)=p_{ij}(c_i+y_i)y_j.
\]
Let $(y_j)_t$ be the nonzero homogeneous component of the
lowest degree part of $y_j$. Then, by comparing the
lowest degree components of the above equation,
we have $c_i (y_j)_t=p_{ij}c_i (y_j)_t$. Thus,
$(y_j)_t=0$ as $A\otimes C$ is a domain, contradiction.

(3) $\Leftrightarrow$ (4) Let $g$ be an automorphism of
$A$. Since $g(x_i)$ has zero constant term, $\gr g\in \Aut_{\gr}(A)$
and $g(\gr g)^{-1}\in \Aut_{\uni}(A)$. Hence (4) is equivalent to (3)
when $C=k$. Then we use the fact that (3) $\Leftrightarrow$ (1), which
is independent of $C$.

(6) $\Rightarrow$ (1) If $x_s$ is central for some $s$, then
$\partial: x_i\to \delta_{is}, c\to 0$ for all $c\in C$ defines
a locally nilpotent derivation
such that the constant term of $\partial(x_s)$ is not zero.

(2) $\Rightarrow$ (5) Suppose $\partial(x_i)=c_i+f_i$,
where $c_i\in C$ is the constant term of $\partial(x_i)$.
Suppose $c_s\neq 0$.
Applying $\partial$ to the equation
$x_i x_s=p_{si} x_s x_i$ for $i\neq s$ we have
\[
(c_i+f_i)x_s+x_i(c_s+f_s)=p_{si}((c_s+f_s)x_i+x_s(c_i+f_i)).
\]
The degree 1 part of the above equation is
\[
c_i x_s+c_s x_i=p_{is}(c_s x_i+c_i x_s).
\]
Since $p_{is}\neq 1$ for some $i$, we have
$c_sx_i+c_i x_s=0$, which contradicts
$c_s\neq 0$. Therefore the assertion holds.
\end{proof}

By Lemma \ref{yylem3.2}(4), to describe $\Aut(A)$, we need understand
both $\Aut_{\gr}(A)$ and $\Aut_{\uni}(A)$.  The next theorem takes care
of $\Aut_{\uni}(A)$ for many cases; this can be viewed as an extension
of results in \cite{AlC}, as we give some necessary and sufficient
conditions so that $\Aut(A)=\Aut_{\gr}(A)$.

Let $(T_s)_{\geq 2}$ be the subset of $T_s$ consisting of elements
$(d_1,\dots, \widehat{d}_s, \dots, d_n)$ with $\sum_j d_j\geq
2$. Recall that $X^W= \{x_1^{d_1}\cdots x_n^{d_n}\mid (d_1, \dots,
d_n)\in W\}$.

Let $C$ be a commutative domain containing $k$ and let $g\in
\Aut_{\uni}(A\otimes C)$.  For each fixed $s$, write
\[
g(x_s)=x_s(1+h')+g_s,
\]
where $g_s$ is in the subalgebra generated by $C$ and
$x_1,\dots,\widehat{x}_s,\dots,x_n$. If $g_s\neq 0$,
it is further decomposed as
\[
 g_s=h_{t_s}+{\text{higher degree terms}},
\]
where $t_s$ is the lowest possible degree of a nonzero
homogeneous component of $g_s$. Define a bigrading
on $g$ by $\deg g=(a, b)$, where
\[
a=\min\{ t_s\mid g_s \neq 0 \ \text{and} \ 1\leq s\leq n\}
\]
and $b=\min\{s\mid t_s=a\}$. If $g_s=0$
for all $s$, then we write $\deg g=(-\infty,-\infty)$. Otherwise,
$\deg g \in \{2,3,4,\dots\}\times \{1,2,\dots,n\}$. For
pairs of integers $(a_1,b_1)$ and $(a_2,b_2)$, we define
$(a_1,b_1)<(a_2,b_2)$ if either $a_1< a_2$ or $a_1=a_2$
and $b_1<b_2$.

\begin{lemma}
 \label{yylem3.3}
Let $g, g_1, g_2\in \Aut_{\uni}(A)$.
\begin{enumerate}
\item
If $\deg g=(-\infty, -\infty)$, then $g$ is the identity.
\item
$\deg g_1 g_2\geq \min\{\deg g_1, \deg g_2\}$ and equality holds
if $\deg g_1\neq \deg g_2$.
\item
If $g$ is not the identity, there is
$(d_1,\dots,\widehat{d}_s, \dots,d_n)\in (T_s)_{\geq 2}$
and $F$, a linear combination of $f_d\in X^{(T_s)_{\geq 2}}$
of the same total degree, such that
$\deg g(F,s)=\deg g$ and $\deg g(F,s) g> \deg g$. \textup{(}Here, $g(F,s)$ is
as defined in \eqref{2.10.1}.\textup{)}
\end{enumerate}
\end{lemma}

\begin{proof}
(1) Since $\deg g=(-\infty, -\infty)$, by definition, $g_s=0$ for all
$s$, or $g(x_s)=x_s(1+h')$, where the constant term of $h'$ is zero.
Since $x_s$ is not a product of two non-units, $g(x_s)\neq
x_s(1+h')$ unless $h'=0$. Thus $g(x_s)=x_s$ for all $s$ and $g$
is the identity.

(2) Left to the reader.

(3) By part (1), $\deg g\neq (-\infty,-\infty)$.
Let $\deg g=(a,s)$. Since $g$ is unipotent,
\[
g(x_i)=x_i (1+h') + h_{t_i}+{\text{higher degree terms}},
\]
where $h_{t_i}$ is the nonzero component of lowest degree that does
not involve $x_i$. By definition, $t_s=a$ and if $h_{t_i} \neq 0$,
then $t_i\geq a$ for all $i$. Note that $h_{t_s}$ is a linear
combination of certain monomials $x_1^{d_1}\cdots\widehat{x}_s \cdots
x_n^{d_n}$. We claim that each $(d_1,\dots,\widehat{d}_s,\dots, d_n)$
is in $(T_s)_{\geq 2}$.  Applying $g$ to the equation $x_i x_s=p_{si}
x_s x_i$ for each $i$ and removing all terms with $x_s$, we obtain
that
\[
 x_i h_{t_s}+{\text{higher degree terms}}
=p_{si} h_{t_s} x_i+{\text{higher degree terms}}.
\]
For any nonzero monomial component $c x_1^{d_1}\cdots
\widehat{x}_s\cdots x_n^{d_n}$ of $h_{t_s}$, the above equation yields
\[
\prod_{j=1, j\neq s}^n p_{ji}^{d_j}=p_{si},
\]
which is the equation defining $T_s$.
Note that $t_s=\sum_j d_j\geq 2$, so
$(d_1,\dots,\widehat{d}_s,\dots,d_n)$ is in $(T_s)_{\geq 2}$. The claim is proved.

Let $F=-h_{t_s}$, which is a linear combination of elements of $f_d\in
X^{(T_s)_{\geq 2}}$ of total degree $t_s$, and then let $g'=g(F,s)
g$. One can show that, for any $i\neq s$, $h'_{t_i}=h_{t_i}$ and
$h'_{t_s}= h_{t_s}-F=0$. By definition, $\deg g'>\deg g=\deg g(F,s)$.
\end{proof}

\begin{theorem}
 \label{yythm3.4} 
 Let $A=k_{p_{ij}}[\underline{x}_n]$ be a skew polynomial ring
 satisfying {\rm{(H1)}}. The following are equivalent.
 \begin{enumerate}
  \item
  Every automorphism of $A$ is affine. Equivalently, $\Aut_{\uni}(A)$
is trivial.
  \item
  For any commutative domain $C$ containing $k$,
  every $k$-algebra automorphism of $A\otimes C$ is $C$-affine.
  \item
$ (T_s)_{\geq 2}=\emptyset$ for all $s$.
\end{enumerate}
If, in addition, $\Z\subset k$, then {\rm{(1)--(3)}}
are also equivalent to the next two.
\begin{enumerate}
\item[(4)]
Every locally nilpotent
derivation of $A$ of nonzero degree is zero.
\item[(5)]
  For any commutative domain $C$ containing $k$,
every locally nilpotent
derivation of $A\otimes C$ of nonzero degree {\rm{(}}with respect
to the $x_i$-grading{\rm{)}} is zero when restricted to $A$.
 \end{enumerate}
\end{theorem}

\begin{proof}
(2) $\Rightarrow$ (1) Trivial.

(1) $\Rightarrow$ (3) Suppose that $(T_s)_{\geq 2}$ is
non-empty for some $s$. Then the system
\eqref{2.9.1} has a solution
\[
(d_1,d_2,\dots,d_{s-1},d_{s+1},\dots,d_n),
\]
where $d_i\geq 0$ and $\sum_i d_i\geq 2$. Let
$f=x_1^{d_1}\cdots x_{s-1}^{d_{s-1}}x_{s+1}^{d_{s+1}}\cdots x_x^{d_n}$;
this has degree at least 2. Then, by \eqref{2.10.1},  the map
\[
g: x_i\to \begin{cases}
             x_i &{\text{if $i\neq s$,}}\\
             x_s+f &{\text{if $i=s$}}
            \end{cases}
\]
extends to a non-affine algebra automorphism of $A$.

(3) $\Rightarrow$ (2)
Let $\mathfrak{m}$ be the graded ideal $A_{\geq 1}\otimes C$.
Suppose that $g$ is a non-$C$-affine automorphism of
$A\otimes C$. Since each $x_i$ is not central,
each $g(x_i)$ has zero constant term [Lemma \ref{yylem3.2}]. Consequently,
$g(x_i)\in \mathfrak{m}$. Thus $g$ preserves the
ideal $\mathfrak{m}$. Using the $\mathfrak m$-adic filtration,
$\gr g$ is a $C$-affine automorphism of $\gr A\otimes C$, which is
isomorphic to $A\otimes C$. Hence $h:= g(\gr g)^{-1}$ is an algebra
automorphism of $A\otimes C$ such that $h|_C=Id_C$,
and $h(x_i)=x_i+
{\text{higher degree terms}}$ for all $i$. That is, $h$ is
a unipotent automorphism of the $C$-algebra $A\otimes C$. Since
$g$ is not $C$-affine, neither is $h$.
The assertion follows from Lemma \ref{yylem3.3}(3) (when
working with the base commutative ring $C$).

(5) $\Rightarrow$ (4) Trivial.

(4) $\Rightarrow$ (3) Suppose that, for some $s$, $(T_s)_{\geq 2}$
is non-empty, containing some element
$(d_1,\dots,\widehat{d}_s,\dots,d_n)$.  Let
$f=x_1^{d_1}\cdots x_{s-1}^{d_{s-1}}x_{s+1}^{d_{s+1}}\cdots
x_x^{d_n}$. Since this has degree at least 2, the map \eqref{2.10.2}
\[
\partial: x_i\mapsto \begin{cases}
             0 &{\text{if $i\neq s$,}}\\
             f &{\text{if $i=s$}}
            \end{cases}
\]
extends to a locally nilpotent derivation of degree at least 2.

(2) $\Rightarrow$ (5) Here we need the hypothesis that
$\Z\subset k$. After localizing, we may assume
that $k$ is a field of characteristic zero.

Let $\partial$ be a nonzero locally
nilpotent derivation of $A\otimes C$.
Let $g_c:=\exp(c\partial)$ for $c\in k$. We know that the
constant term of $g_c(x_i)$ is zero for all $i$ and $c$.
Then the constant term of $\partial^n(x_i)$ is zero for all
$n$. If the degree of $\partial$ is not zero, then $g_c$ is
not $C$-affine, a contradiction.
\end{proof}

An immediate consequence of
Lemma \ref{yylem2.9} and Theorem \ref{yythm3.4} is:
if $C(k_{p_{ij}}[\underline{x}_n^{\pm 1}])
\subset k\langle x_1^{\pm \alpha_1},\dots, x_n^{\pm \alpha_n}
\rangle$
for some $\alpha_1,\dots, \alpha_n\geq 2$,
then $\Aut_{\uni}(A)$ is trivial.

The following is easy to check.

\begin{lemma}
\label{yylem3.5} Assume {\rm{(H2)}}. Then
$T_s=\emptyset$ if and only if $T_s$ is finite if and only if
$(T_s)_{\geq 2}=\emptyset$.
\end{lemma}

The next theorem is a version of Theorem \ref{yythm3.4} when (H1) is
replaced by (H2). Note that this is part of Theorem \ref{yythm3.1}.
Its proof is similar to the proof Theorem \ref{yythm3.4} and therefore
is omitted. Let $\Aut_{\uni\text{-}C}(A\otimes C)$ be the set of
$k$-algebra automorphisms $g$ of $A\otimes C$ such that $g|_C=Id_C$
and $g(x_i)=x_i+{\text{higher degree terms}}$ for all $i$.

\begin{theorem}
 \label{yythm3.6}  
 Let $A=k_{p_{ij}}[\underline{x}_n]$ be a skew polynomial ring
 satisfying {\rm{(H2)}}. The following are equivalent.
 \begin{enumerate}
  \item
 $\Aut_{\uni}(A)=\{1\}$
  \item
 For any commutative domain $C$ containing $k$,
$\Aut_{\uni\text{-}C}(A\otimes C)=\{1\}$.
  \item
$ T_s=\emptyset$ for all $s$.
\end{enumerate}
If $\Z\subset k$, then the above is also equivalent to
\begin{enumerate}
\item[(4)]
$\LNDer(A)=\{0\}$.
\item[(5)]
 For any commutative domain $C$ containing $k$,
every locally nilpotent
derivation $\partial$ of $A\otimes C$ with $\partial|_C=0$  is zero.
 \end{enumerate}
\end{theorem}

By Theorem \ref{yythm3.4} or Theorem \ref{yythm3.6}, if $A$ is the
algebra in Example \ref{yyex2.4}(4), then it is easy to check that
each $T_s=\emptyset$, so $\Aut(A)$ is affine. (Alternatively, one can
apply Lemma \ref{yylem2.9}.) Here is another example.

\begin{example}
\label{yyex3.7} Let $n=4$ and $i^2=-1$. Let
\[
p_{12}= i,\quad
p_{13}=i,\quad
p_{14}= i,\quad
p_{23}=-i,\quad
p_{24}=i,\quad
p_{34}=1.
\]
Then $A:=k_{p_{ij}}[x_1,x_2,x_3,x_4]$ is a PI algebra with its center
generated by $x_i^4$, $x_1^2x_2^2x_3^2$, $x_1^2x_2^2x_4^2$ and
$x_3^2x_4^2$. Therefore $C(A)$ is not isomorphic to the polynomial
ring; in fact, the center is not Gorenstein. One can check directly
that $C(k_{p_{ij}} [\underline{x}_n^{\pm 1}]) \subset k\langle
x_1^{\pm 2},\dots, x_4^{\pm 2} \rangle$. Therefore $\Aut(A)$ is affine
by Lemma \ref{yylem2.9} and Theorem \ref{yythm3.4}.
\end{example}

Along these lines, here is another part of Theorem \ref{yythm3.1}.

\begin{theorem}
\label{yythm3.8}
Let $A=k_{p_{ij}}[\underline{x}_n]$ be a skew polynomial ring
satisfying {\rm{(H2)}}. The following are equivalent.
\begin{enumerate}
\item
$\Aut(A)$ is affine.
\item
$C(k_{p_{ij}}[\underline{x}_n^{\pm 1}])
\subset k\langle x_1^{\pm \alpha_1},\dots, x_n^{\pm \alpha_n}
\rangle$
for some $\alpha_1,\dots, \alpha_n\geq 2$.
\item
$C(k_{p_{ij}}[\underline{x}_n])
\subset k\langle x_1^{\alpha_1},\dots, x_n^{\alpha_n}
\rangle$
for some $\alpha_1,\dots, \alpha_n\geq 2$.
\item
$T_s=\emptyset$ for all $s=1,\dots,n$.
\end{enumerate}
\end{theorem}

\begin{proof}
(1) $\Rightarrow$ (4) If $T_s\neq \emptyset$, then $(T_s)_{\geq 2}\neq
\emptyset$. Hence, picking some element in $(T_s)_{\geq 2}$, the
construction 
\eqref{2.10.1} defines a non-affine automorphism of $A$.

(4) $\Rightarrow$ (2) Let $b_s$ be the smallest positive integer such
that $x_s^{b_s}$ is in the center of $A$ (and in the center of
$k_{p_{ij}}[\underline{x}_n^{\pm 1}]$). Since all $p_{ij}$ are roots
of unity, $b_s$ exists for each $s$.

For each $s$, let $a_s$ be the smallest positive integer such that
$x_1^{a_1}\cdots x_s^{a_s}\cdots x_n^{a_n}\in
C(k_{p_{ij}}[\underline{x}_n^{\pm 1}])$ for some $a_i$. Then every
monomial in the center is of the form $x_1^{c_1}\cdots x_s^{c_s}\cdots
x_n^{c_n}$, where $a_s\mid c_s$.  Suppose the assertion in (2)
fails. Then $a_s=1$ for some $s$.  By multiplying by $x_i^{-b_i}$ if
necessary, we may assume that there are $a_i>1$ for all $i\neq s$ such
that $x^{-a_1}\cdots x_s\cdots x_n^{-a_n}\in
C(k_{p_{ij}}[\underline{x}_n^{\pm 1}])$. Equivalently, $x^{a_1}\cdots
x_s^{-1}\cdots x_n^{a_n}\in C(k_{p_{ij}}[\underline{x}_n^{\pm
1}])$. Thus $T_s\neq \emptyset$ by Lemma \ref{yylem2.9}(2).

(2) $\Rightarrow$ (3) Clear.

(3) $\Rightarrow$ (4) Follows by Lemma \ref{yylem2.9}(2).

(4) $\Rightarrow$ (1) From the earlier parts, we know that (4)
$\Rightarrow$ (3); therefore no $x_i$ is central in $A$: $A$ satisfies
(H1). The proof now follows by Lemma \ref{yylem3.5} and Theorem
\ref{yythm3.4}.
\end{proof}

Now we can prove Theorem \ref{yythm3.1}.

\begin{proof}[Proof of Theorem \ref{yythm3.1}]
The equivalences of
(1)--(4) are given in Theorem \ref{yythm3.8}.

(4) $\Rightarrow$ (5) This is Theorem \ref{yythm2.11}(3).

(5) $\Rightarrow$ (6) Trivial.

(6) $\Rightarrow$ (1) This is Theorem \ref{yythm1.13}(1).

(7) $\Leftrightarrow$ (4) is given in Theorem \ref{yythm3.6}.
\end{proof}

The next proposition takes care of $\Aut_{\gr}(A)$ in many
cases.

\begin{proposition}
\label{yypro3.9}
Suppose that $p_{ij}\neq 1$ for all $i<j$. Let
\[
S=\{\sigma\in S_n\mid
p_{ij}=p_{\sigma(i)\sigma(j)} \quad
{\text{for all $i,j$}}\}.
\]
\begin{enumerate}
 \item
Then
\[
\Aut_{\gr}(A)=S\ltimes (k^\times)^n.
\]
 \item
Suppose the conditions in Theorem {\rm{\ref{yythm3.4}}(1)--(3)} hold.
Then
\[
\Aut(A)=\Aut_{\gr}(A)=S\ltimes (k^\times)^n.
\]
If, further, $\Z\subset k$, then
every locally nilpotent derivation of $A$ is zero.
\end{enumerate}
\end{proposition}

\begin{proof} (1) It is clear that $S\ltimes (k^\times)^n\subset
\Aut_{\gr}(A)$. We claim that $S\ltimes (k^\times)^n\supset
\Aut_{\gr}(A)$.
Since $p_{ij}\neq 1$, every graded automorphism $g$
of $A$ is of the form $g: x_i\to c_i x_{\sigma(i)}$ for some
$c_i\in k^\times$ and $\sigma\in S_n$ \cite[Lemma 2.5(e)]{KKZ3}.
Then $\sigma\in S$. The claim is proved.

(2) $\Aut_{\uni}(A)$ is trivial by assumption, so 
$\Aut(A)=\Aut_{\gr}(A)$ by Lemma \ref{yylem3.2}.

The assertion about locally nilpotent derivations follows
from a similar argument in the proof of (2) $\Rightarrow$ (5) in
Theorem \ref{yythm3.4}.
\end{proof}

In the following special case, $\Aut(A)$ being  affine is equivalent to
$C(A)$ being isomorphic to a polynomial ring.

\begin{theorem}
\label{yythm3.10} 
Let $A=k_{p_{ij}}[\underline{x}_n]$ be a skew polynomial ring
satisfying {\rm{(H1)}} and {\rm{(H2)}}. Suppose that
the subgroup of $k^\times$ generated by parameters
$\{p_{ij}\mid i<j\}$ is equal to $\langle q\rangle$ where $\ell$ is
prime and $q$ is a primitive $\ell$th root of unity.
Then the following are equivalent.
\begin{enumerate}
\item
$\Aut(A)$ is affine.
\item
For any commutative domain $C$ which is $k$-flat, every $k$-algebra
automorphism of $A\otimes C$ is $C$-affine.
\item
$(k^\times)^n\subset \Aut(A)\subset S_n\ltimes (k^\times)^n$,
where $(k^\times)^n$ is viewed as $\Aut_{\Z^n-\gr}(A)$.
\item
$C(A)$ is isomorphic to a polynomial ring.
\item
$A$ is a free module over $C(A)$.
\item
The determinant $\det (\phi_{ij})_{n\times n}$ is invertible in
$\Z/(\ell)$,
where $\phi_{ij}$ are determined by $p_{ij}=q^{\phi_{ij}}$
for all $i$ and $j$.
\item
$d_w(A/C(A))$ is dominating where $w=\rk(A/C(A))$.
\end{enumerate}
If $\Z\subset k$, then the above are also equivalent to
\begin{enumerate}
\item[(8)]
$\Aut(A[t])=\Aut_{\tr}(A[t])$.
\item[(9)]
Every locally nilpotent derivation of $A$ is zero.
\end{enumerate}
\end{theorem}

\begin{proof}
The equivalence of (1), (7) and (9) is given in
Theorem \ref{yythm3.1}.

(1) $\Rightarrow$  (3) If $\Aut(A)$ is affine, then
$\Aut(A)=\Aut_{\gr}(A)$. The assertion follows
from Proposition \ref{yypro3.9}(1).

(3) $\Rightarrow$  (1) Part (3) says that
there are no non-trivial unipotent automorphisms.
Hence every automorphism is affine by Lemma \ref{yylem3.2}(4).

(1) $\Rightarrow$ (4) If $C(A)$ is not a polynomial ring,
by Lemma \ref{yylem2.5}, there is a solution
$(d_1,\dots, d_n)\in \N^n$ to the system of equations
\[
\prod_{j=1}^n p_{ij}^{d_j}=1, \quad {\text{for all $i$}}
\]
with $d_s=1$. So for each $i$,
\[
p_{is}=p_{si}^{-1}=\prod_{j=1,j\neq s}^n p_{ij}^{-d_j}
=\prod_{j=1,j\neq s}^n p_{ij}^{-d_j+a\ell}
\quad {\text{for any $a\in \N$}}.
\]
Hence, for some $a>0$, $(-d_1+a\ell, \cdots,
-d_{s-1}+a\ell, -d_{s+1}+a\ell,
\cdots, -d_n+a\ell)\in \N^{n-1}$ is a solution to
\eqref{2.9.1} with $\sum_{j\neq s}
(-d_j+a\ell)\geq 2$. Thus $\Aut(A)$ is not affine, a contradiction.
Therefore $C(A)$ is a polynomial ring.

(4) $\Leftrightarrow$ (5) Lemma \ref{yylem2.3}.

(4) $\Rightarrow$ (6) follows from Lemma \ref{yylem2.5} and
linear algebra.

(6) $\Rightarrow$ (4) Example \ref{yyex2.4}(1).

(5) $\Rightarrow$ (7) Proposition \ref{yypro2.8}.

(7) $\Rightarrow$ (2) \cite[Lemma 3.2(1)]{CPWZ}.

(2) $\Rightarrow$ (1) is obvious.

(7) $\Rightarrow$ (8) Theorem \ref{yythm1.13}.

(8) $\Rightarrow$ (9) \cite[Lemma 3.3]{CPWZ}.
\end{proof}

Note that part (1) does not imply part (4) if $\ell$ is 4 (which is
not prime) -- see Example \ref{yyex3.7}.

Here are some cases in which the hypotheses of Proposition \ref{yypro3.9}
hold.

(1) Assume $p_{ij}=q$ for all $i<j$ and $q$ is not a root of unity. For any
fixed $s$ between 1 and $n$, the condition
\eqref{2.9.1} says, in this case, that
for any $i<s$,
\begin{equation}
\label{3.10.1}
\sum_{j<i} d_j-\sum_{j>i, j\neq s} d_j=-1
\end{equation}
and, for any $i>s$,
\begin{equation}
\label{3.10.2}
\sum_{j<i, j\neq s} d_j-\sum_{j>i} d_j=1
\end{equation}
for non-negative integers $(d_1,\dots, \widehat{d}_s,\dots,d_n)$.
If $n=2$, it is easy to check that there is no solution.
If $n=3$, there is one solution when $s=2$, which is
$(d_1,d_3)=(1,1)$. As in \cite[Theorem 1.4.6(i)]{AlC}, we have
\[
\begin{aligned}
\Aut&(k_q[x_1,x_2,x_3])\\
&=\{g: x_1\mapsto a_1x_1, \, x_2\mapsto a_2x_2+b x_1x_3, \, x_3\mapsto a_3 x_3,
\ {\text{where $a_i\in k^\times$, $b\in k$}}\}.
\end{aligned}
\]

Assume now $n\geq 4$. When $s=1$, taking $i=2$, \eqref{3.10.2} becomes
$-\sum_{j>2}d_j=1$, which has no solution.  Similarly, \eqref{3.10.1} has
no solution for $s=n$. When $1<s<n$, take $i=1$ and $i=n$; then
\eqref{3.10.1} and \eqref{3.10.2} (and the condition that $\sum_j d_j\geq 2$)
imply that $d_1=d_n=1$ and
$d_j=0$ for all $1<j<n, j\neq s$. Since $n\geq 4$, there is another $i$
with $i\neq 1,s,n$. Then either \eqref{3.10.1} or
\eqref{3.10.2} gives a contradiction. In summary, we recover
\cite[Theorem 1.4.6(ii)]{AlC}, which states that
$\Aut(k_q[x_1,\dots,x_n]) =(k^\times)^n$ if and only if $n=2$ or $n\geq 4$.

(2) If $p_{ij}=q$ for all $i<j$ and $q$ is a root of unity, then Example 1.8 shows that
$\Aut(A)$ is not affine when $n$ is odd. But one can check
by using Proposition \ref{yypro3.9} that if $n$ is even, then
$\Aut(A)$ is affine. We will give another proof of this fact later.

\begin{theorem}
\label{yythm3.11}
Let $s_0$ be an integer between $1$ and
$n$. Suppose that $T_s$ is finite
for all $s\neq s_0$. Then every unipotent automorphism
$g$ is a product of elementary automorphisms:
\[
g=g(F_1,s_{n_1}) g(F_2,s_{n_2})\cdots g(F_w,s_{n_w}).
\]
Moreover, we may choose that the degrees $\deg g(F_i, s_{n_i})$ are strictly
increasing, or alternatively, strictly decreasing. In either case,
the decomposition is unique.
\end{theorem}

\begin{proof}
We will construct the factorization and show that the degrees $\deg
g(F_i,s_{n_i})$ are strictly increasing. Replacing $g$ by $g^{-1}$, we
obtain the case when the degrees are strictly decreasing.

We use downward induction on $\deg g$. By the hypothesis
that every $T_s$ except possibly $T_{s_0}$ is finite, we
first assume that $\deg g=(a,s)$ and $a> \sum_j d_j$ for any
$(d_1,\dots,\widehat{d}_s,\dots,d_n)\in \bigcup_{s\neq s_0}T_s$.
For any $i\neq s_0$, if
\[
g(x_i)=x_i(1+h')+h_{t_i}+{\text{(higher degree terms)}},
\]
then the proof of Lemma \ref{yylem3.3}(3) shows that
$h_{t_i}=0$ and $g(x_i)=x_i(1+h')$. Since $x_i$ is not a product
of two non-units, $g(x_i)=x_i$ for all $i\neq s_0$.

Now let $g(x_{s_0})=x_{s_0}+
{\text{(higher degree terms)}}$ and write $g(x_{s_0})=\sum f_i x_{s_0}^i$
and $g^{-1}(x_{s_0})=\sum_j h_j x_{s_0}^j$
with $f_i, h_j\in k_{p_{ij}}[x_1,\dots, \widehat{x}_{s_0},\dots,x_n]$.
Suppose $m,q$ are maximal so that $f_m h_q\neq 0$. Then
\[
x_{s_0}=gg^{-1}(x_{s_0})
=h_q f_{m}^{q} x_{s_0}^{mq} + \text{(lower degree terms)}.
\]
Thus $m=q=1$ and $h_1\in k^\times$. Similarly,
$f_1\in k^\times$. Since $g$ and $g^{-1}$ are unipotent,
$f_1=h_1=1$. This means that $g(x_{s_0})=x_{s_0}+\sum_f c_f f$,
where $f$ ranges over a set of monomials in
$k_{p_{ij}}[x_1,\dots, \widehat{x}_{s_0},\dots,x_n]$.
Now the argument in the proof of Lemma \ref{yylem3.3}(3)
shows that $f\in T_{s_0}$ for all $f$. In this case, it is easy to see that
$g$ can be decomposed into the form as given,
\[
g=g(F_1, s_0)g(F_2, s_0)\cdots g(F_w, s_0),
\]
and these $g(F_i, s_0)$ commute. Uniqueness follows
from the fact that each $F_i$ is precisely a homogeneous component
of $f$.

Next assume that $g$ is not the identity. By Lemma \ref{yylem3.3}(3),
$\deg g(F,s)g>\deg g$ for some $F$ and $s$.
By the inductive hypothesis, $g(F,s)g=g(F_2, s_2)\cdots g(F_w, s_w)$.
Then $g=g(-F,s)g(F_2, s_2)\cdots g(F_w, s_w)$. Let $F_1=-F$ and $s_1=s$.
The uniqueness of $(F_1,s_1)$ can be read off from the proof of
Lemma \ref{yylem3.3}(3) and the fact that $\deg g(F_i,s_i)$ are
increasing. The inductive hypothesis also says that the $(F_i,s_i)$ are unique
for $i>1$. The assertion now follows.
\end{proof}

\begin{proof}[Proof of Theorem \ref{yythm0.4}]
Let $g$ be in $\Aut(A)$. Since $A$ satisfies (H1), $g(x_i)$
has no constant term by Lemma \ref{yylem3.2}.
Then the associated graded map $\gr g$ is a graded
(hence affine) automorphism of $A$. Now $g (\gr g)^{-1}$ is
a unipotent automorphism. The assertion follows from the equation
$g=[g(\gr g)^{-1}] (\gr g)$ and Theorem \ref{yythm3.11}.
\end{proof}

To conclude this section we give some examples.

\begin{example}
 \label{yyex3.12}
\begin{enumerate}
\item
Let $q$ be a primitive $\ell$th root of unity and $\ell=abc$, where
$a,b,c\geq 2$ are pairwise coprime.
If $p_{12}=q^a$, $p_{13}=q^b$ and $p_{23}=q^c$ and $A=k_{p_{ij}}[x_1,
x_2, x_3]$,
then one can check that $T_s=\emptyset$ for $s=1,2,3$ and
$\Aut(A)=(k^\times)^3$. Similar statements can be made if there are
more than three generators.
\item
If $A=k_{p_{ij}}[x_1, x_2, x_3]$ is not PI, then it is easy to
check that each $T_s$ is finite for $s=1,2,3$.
As a consequence of Theorem \ref{yythm3.11},
$\Aut(A)$ is tame.  Here is an explicit example.
Assume that $q$ is not a root of unity. Let $p_{12}=q^m$, $p_{13}=q$ and
$p_{23}=q^n$ for some integers $m,n\geq 1$. Then $T_1=\emptyset$,
$T_2=\{(n, \widehat{d}_2,m)\}$ and $T_3=\emptyset$. Hence every automorphism of
$k_{p_{ij}}[x_1,x_2,x_3]$ is of the form
\[
\begin{aligned}
x_1 &\mapsto a_1 x_1,\\
x_2&\mapsto a_2 x_2 +b x_1^n x_3^m,\\
x_3&\mapsto a_3x_3,
\end{aligned}
\]
where $a_i\in k^\times$ and $b\in k$ [Theorem \ref{yythm3.11}].
This should be compared with \cite[Theorem 1.4.6(i)]{AlC}.
\end{enumerate}
\end{example}

\begin{example}
 \label{yyex3.13}
\begin{enumerate}
\item \cite[Example 3.8]{CPWZ}
If $p_{12}=1$, $p_{13}=q$, $p_{23}=q$, where $q$ is not a root of unity,
then the system of equations \eqref{2.9.1} for $s=1$ and for $s=2$ has only
one solution $(d_2,d_3)=(1,0)$, and the system of equations for $s=3$
has no solution. Therefore these systems of equations have no solution
with $\sum_j d_j\geq 2$. By Theorem \ref{yythm3.4}, every automorphism
of $A=k_{p_{ij}} [x_1,x_2,x_3]$ is affine.
\item
By the analysis of the case $n=2$, every automorphism of
$B=k_p[x_4,x_5]$ (when $p\neq 1$ and $p^w=1$) is affine.
\item
The tensor product $C=A\otimes B$ is a skew polynomial
ring $k_{p_{ij}}[x_1,\dots,x_5]$. But $C$ has a non-affine
automorphism determined by
\[
\begin{aligned}
  g(x_1)&= x_1+ x_2 x_4^w x_5^w,\\
  g(x_i)&=x_i, \quad {\text{for all $i>1$}}.
  \end{aligned}
\]
\end{enumerate}
So even if $A$ and $B$ only have affine automorphisms, $A\otimes B$
may have non-affine automorphisms. Compare this with
\cite[Theorem 5.5]{CPWZ}.
\end{example}

\section{Miscellaneous operations and constructions}
\label{yysec4}

In this section we discuss some general methods that deal with
automorphisms and discriminants, for use in proving Theorem
\ref{yythm0.2}. Two examples: in Subsection \ref{yysec4.1} we develop
tools to study automorphisms of tensor products of algebras. In
Subsection \ref{yysec4.4} we look at filtered algebras: if $B$ is
filtered and $C$ is a central subalgebra of $B$; then with some extra
hypotheses, $\gr d_{w}(B/C) = d_{w}(\gr B/ \gr C)$ (Proposition
\ref{yypro4.10}).

\subsection{Tensor products and the categories $\Aff_{-s}$ and $\Af_{-s}$}
\label{yysec4.1}

Recall from \cite[Definition 2.4]{CPWZ} that the category
$\Af$ consists of all $k$-flat algebras $A$ satisfying the
following conditions:
\begin{enumerate}
\item
$A$ is an algebra with standard filtration such that the
associated graded ring $\gr A$ is a connected graded domain,
\item
$A$ is a finitely generated free module over its center $C(A)$, and
\item
the discriminant $d(A/C(A))$ is dominating.
\end{enumerate}
The morphisms in this category are isomorphisms of algebras.

We extend this definition to a more general situation.

\begin{definition}
\label{yydef4.1}
Let $s$ be a non-negative integer.
\begin{enumerate}
\item
Let $\Aff_{-s}$ be the category
consists of all $k$-flat algebras $A$ satisfying the
following conditions:
\begin{enumerate}
\item
$A$ is an algebra with standard filtration such that the
associated graded ring $\gr A$ is a connected graded domain,
\item
the $w$-discriminant $d_w(A/C(A))$ is $(-s)$-dominating where
$w$ is the rank $\rk(A/C(A))$.
\end{enumerate}
\item
Let $\Af_{-s}$ be the category
consists of all $k$-flat algebras $A$ satisfying the
following conditions:
\begin{enumerate}
\item
$A$ is in $\Aff_{-s}$, and
\item
$A$ is a finitely generated free module over its center.
\end{enumerate}
\end{enumerate}
The morphisms in these categories are isomorphisms of algebras.
\end{definition}

\begin{remark}
\label{yyrem4.2}
\begin{enumerate}
\item
$\Af=\Af_0$.
\item
$\Af_{-s}$ is a full subcategory of $\Aff_{-s}$ for any $s$.
\item
If $A$ is in $\Aff_{-s}$, then every automorphism
of $A$ is $(-s)$-affine; see the proof of
\cite[Lemma 2.6]{CPWZ}.
\end{enumerate}
\end{remark}

Let $A$ be a noncommutative domain. Let $\mathcal{D}:=
\{d_i\}_{i\in I}$ be a subset of $A$ with gcd $y$. Let
$\mathcal{D}^n$ denote $\{d_{i_1} \cdots d_{i_n}
\mid d_{i_s}\in \mathcal{D}\} \subseteq A$. Let $A'$ be
another domain. We say $\mathcal{D}$ is \emph{$A'$-saturated}
if for every positive integer $n$
and every $0\neq f\in A'$, the subset
$\mathcal{D}^n\otimes f$ in $A\otimes A'$ has gcd
$y^n\otimes f$ (also written as $y^n f$).
If $\mathcal{D}$ is a subset of monomials in
$k_{p_{ij}}[\underline{x}_n]$, then $\mathcal{D}$ is
$A'$-saturated for any domain $A'$.

\begin{lemma}
\label{yylem4.3}
Let $A$ and $A'$ be two domains with generating sets $X$ and $X'$ and
with semi-bases $b$ and $b'$ over their centers $C(A)$ and $C(A')$.
Suppose that $C(A)$ and $C(A')$ are $k$-flat and that $b'$ is a quasi-basis.
Let $m=\rk(A/C(A))$ and $n=\rk(A'/C(A'))$. Let $w=mn$.
\begin{enumerate}
\item
The discriminant $d_w(A\otimes A'/C(A\otimes A'))$ is the
gcd of elements in $\mathcal{D}(X/b)^n \otimes
\mathcal{D}(X'/b')^m$.
\item
Suppose that $A'$ is free over $C(A')$ and that $\mathcal{D}(X/b)$ is
$A'$-saturated. Then $d_w(A\otimes A'/C(A\otimes A'))=d_m(A/C(A))^n
d_n(A'/C(A'))^m$.
\end{enumerate}
\end{lemma}

\begin{proof} (1) Since $b$ and $b'$ are semi-bases of $A$ and $A'$
respectively, $b\otimes b'$ is a semi-basis of $A\otimes A'$ over its
center. Also $X\otimes X'$ is a generating set of $A\otimes A'$
over its center. For each subset
$Z:=\{x_{i_s}\otimes x'_{j_s}\}_{s=1}^w$ of $X\otimes X'$ with
$\det(Z: b\otimes b')\neq 0$, since $b'=\{b'_1, \dots, b'_n \}$ is a
quasi-basis, one can rewrite $Z$ as
\[
Z=\{x_{i_{s,t}}\otimes c_{s,t} b'_t\mid 1\leq s\leq m, 1\leq t\leq n\},
\]
where $c_{s,t}\in C'_t$ (where $C'_t$ is the set of nonzero elements
of the form $(x': b'_t)$, as in Definition~\ref{yydef1.10}). Let
$\widehat{Z}$ be the subset
\[
\quad \{x_{i_{s,t}}\otimes b'_t\mid 1\leq s\leq m, 1\leq t\leq n\}.
\]
Then $\det(Z:\widehat{Z})=\prod_{s,t} c_{s,t}$, which is in
$[(C'_1)\cdots (C'_n)]^m=(X'/b')^m$. By linear algebra,
\[
\det (\widehat{Z}:b\otimes b')=\pm \prod_{t=1}^n \det
(\{x_{1,t},\dots,x_{m,t}\}: b)\in (X/b)^n.
\]
By the proof of \cite[Lemma 5.3]{CPWZ},
$d_w(b\otimes b':tr)=d_m(b:\tr)^n d_n(b':\tr)^m$. For any two
subsets $Z_1$ and $Z_2$ of $X\otimes X'$, we have
\[
d_w(Z_1,Z_2:\tr)= \pm \alpha_1 \alpha_2 \beta_1 \beta_2
d_w(b\otimes b':\tr),
\]
where $\alpha_1$ and $\alpha_2$ are the product of $n$ elements
of the form $\det (\{x_{i_1},\dots,x_{i_m}\}: b)$ and $\beta_1$
and $\beta_2$ are the product of $m$ elements of the form
$\prod_{s=1}^n c_s$. Therefore
$d_w(Z, \widehat{Z}:\tr)$ is in $\mathcal{D}(X/b)^n\otimes
\mathcal{D}(X'/b')^m$. The assertion follows.

(2) If $A'$ is free over $C(A')$, we take $X'=b'$ to be a basis
of $A'$. In this case, $\mathcal{D}(X'/b')$ is a singleton
$\{y\}$, where $y=d_n(A'/C(A'))$. By part (1), the $w$-discriminant
of $A\otimes A'$ over its center is the gcd of
$\mathcal{D}(X/b)^n\otimes y^m$. The assertion follows from
the $A'$-saturatedness of $\mathcal{D}(X/b)$.
\end{proof}

\begin{lemma}
\label{yylem4.4}
Let $s$ and $t$ be non-negative integers.
Assume that $A$ and $B$ are $k$-flat filtered algebras
such that $\gr A\otimes \gr B$ is a connected graded domain.
In part {\rm{(4)}} we also assume that the center of $A$
is $k$-flat.
\begin{enumerate}
\item
If $A\in \Af_{-s}$ and $B\in \Af_{-t}$, then
$A\otimes B\in \Af_{-(s+t)}$.
\item
If $A\in \Aff_{-s}$, then
$A[t]\in \Aff_{-(s+1)}$.
\item
If $A\in \Af_{-s}$, then
$A[t]\in \Af_{-(s+1)}$.
\item
Suppose $A$ is in $\Aff_{-s}$ and $B$ is in
$\Af_{-t}$. If there is a generating set $X$ of $A$ containing a semi-basis
$b$ such that $\mathcal{D}(X/b)$ is $B$-saturated, then
$A\otimes B$ is in $\Aff_{-(s+t)}$.
\end{enumerate}
\end{lemma}

\begin{proof} (1) By hypothesis, $\gr(A\otimes B)\cong
\gr A\otimes \gr B$
is a connected graded domain. It is clear that $A\otimes B$
is finitely generated free over its center
$C(A)\otimes C(B)$. It remains to show that the discriminant
is $(-(s+t))$-dominating.
By \cite[Lemma 5.3]{CPWZ},
$d(A\otimes B/C(A\otimes B))=d(A/C(A))^n d(B/C(B))^m$,
where $m=\rk (A/C(A))$ and $n=\rk(B/C(B))$.
When $d(A/C(A))$ is $(-s)$-dominating and
$d(B/C(B))$ is $(-t)$-dominating, it follows from the
definition that $d(A/C(A))^n d(B/C(B))^m$ is
$(-(s+t))$-dominating.

(2) Let $Y=\bigoplus_{i=1}^n k x_i$ be the generating space of $A$
as in Definition \ref{yydef1.6}. Then $Y'=Y\oplus kt$ is a
generating space of $A[t]$. By Lemma \ref{yylem1.12}(2),
$d_w(A[t]/C(A[t]))=d_w(A/C(A))$. If $d_w(A/C(A))$ is
$(-s)$-dominating with respect to $Y$, then it is
$(-(s+1))$-dominating with respect to $Y'$.

(3) This is a special case of (2).

(4) By Lemma \ref{yylem4.3}(2), $d_w(A\otimes B/C(A\otimes B))=
d_m(A/C(A))^n d_n(B/C(B))^m$. Then the proof of
part (1) works.
\end{proof}

One immediate application of Lemma \ref{yylem4.4}(4) is the
following: assume $A:=k_{p_{ij}}[\underline{x}_n]$ satisfies
(H2). Suppose that $C(A)\subset k\langle x_1^{a_1},\dots,
x_n^{a_n}\rangle$, where $a_i\geq 2$ for all $i$. By Theorem
\ref{yythm3.1}, $A$ is in $\Aff_0$. Let $B$ be the
algebra $k\langle x,y\rangle/(y^2x-xy^2,yx^2+x^2y)$ given
in \cite[Example 5.1]{CPWZ}, which is in $\Af=\Af_0$. By
Lemma \ref{yylem4.4}(4), $A\otimes B$ is in $\Aff_0$.
Then $\Aut(A\otimes B)$ is affine by Theorem \ref{yythm1.13}.

We generalize the notion of elementary automorphisms
of skew polynomial rings, as in the
introduction and \eqref{2.10.1}, as follows.
Suppose that $Y:=\bigoplus_{i=1}^n kx_i$ generates $A$ and let
$g\in \Aut(A)$. We say that $g$ is \emph{elementary} if there is
an $s$ such that $g(x_i)=x_i$ for all $i\neq s$.
In general elementary automorphisms are relatively easy to
determine when all relations of $A$ are understood.

An automorphism $g\in \Aut(A)$ is called \emph{tame}
if it is generated by affine and elementary automorphisms, and
a subgroup $G$ of $\Aut(A)$ is \emph{tame} if every $g$ in
$G$ is tame. Let $A$ be a connected graded algebra. Recall that
$g\in \Aut(A)$ is called unipotent if $g(x)-x$ is a linear
combination of homogeneous elements of degree at least 2 for all $x
\in A_1$.

\begin{proposition}
\label{yypro4.5} Suppose that $A$ is a graded domain generated
by $Y:=\bigoplus_{i=1}^n kx_i$ in degree 1 and that $A$ is in
$\Aff_{-1}$. Then every unipotent automorphism is
elementary. If, further, for every automorphism $h\in\Aut(A)$,
$h(x_i)$ has no constant term, then $\Aut(A)$ is tame.
\end{proposition}

\begin{proof} Let $g$ be a unipotent automorphism and
write $g(x_i)=x_i+f_i$, where $f_i$ is a linear combination
of homogeneous elements of degree at least 2. Since the discriminant
is $(-1)$-dominating, $g$ is $(-1)$-affine by Remark \ref{yyrem4.2}(3).
Hence $\deg g(x_i)\leq 1$ for all but one $i$. Thus $g(x_i)=x_i$
for all but one $i$. Therefore $g$ is elementary.

If $h(x_i)$ has no constant term, then $\gr h$ is a graded
automorphism (hence affine) and $h (\gr h)^{-1}$ is a unipotent
automorphism. The final assertion follows the equation
$h=[h(\gr h)^{-1}] (\gr h)$.
\end{proof}

A special case is when $A$ is the algebra $k_{p_{ij}}[\underline{x}_n]$
that satisfies (H1) and (H2). By Lemma \ref{yylem3.2}(4), every
automorphism of $A$ is generated by graded and unipotent
automorphisms. If $A$ is in $\Aff_{-1}$, then $T_s=\emptyset$
except for one $s$ [Theorem \ref{yythm2.11}(2)]. By the proof
of Theorem \ref{yythm3.11}, every unipotent automorphism is
of the form \eqref{2.10.1}. Applying the above to the algebra
$D$ in Example \ref{yyex1.3}(4),  we obtain  that every automorphism
of $D$ is determined by
\[
g(x_s)=\begin{cases} a_1x_1+\sum_{i,j}b_{ij} x_2^{2+4i}x_3^{2+4j}, &s=1,\\
a_sx_s, & s=2, 3,\end{cases}
\]
where $a_s\in k^\times$ and $b_{ij}\in k$ for all $s$, $i$ and $j$.

\subsection{Mod-$p$ reduction}
\label{yysec4.2}

In this subsection we introduce a general method that deals with
automorphisms of certain non-PI algebras.  Let $K$ be a commutative
domain. We write $\Aut_{\af}(A)$ for the set of affine automorphisms
of an algebra $A$.

\begin{lemma}
\label{yylem4.6} Let $K$ be finitely generated over $\Z$.
Suppose $S$ is a filtered $K$-algebra such that
$\gr S$ is locally finite and connected graded. Suppose
that
$\gr S$ is a free $K$-module, namely, each $(\gr S)_i$ is free over $K$.
\begin{enumerate}
\item
If, for every quotient field $F\cong K/\mathfrak{m}$,
$\Aut(S\otimes_K F)=\Aut_{\af}(S\otimes_K F)$, then
$\Aut(S)=\Aut_{\af}(S)$.
\item
If, for every quotient field $F\cong K/\mathfrak{m}$,
$\Aut(S\otimes_K F[t])=\Aut_{\tr}(S\otimes_K F[t])$, then
$\Aut(S[t])=\Aut_{\tr}(S[t])$.
\item
If, for every quotient field $F\cong K/\mathfrak{m}$,
every locally nilpotent derivation of $S\otimes_K F$ is zero,
then every locally nilpotent derivation of $S$ is zero,
\end{enumerate}
\end{lemma}

\begin{proof}
For every quotient field $F\cong K/\mathfrak{m}$, $S\otimes_K F$
is filtered and $\gr (S\otimes_K F)$ is naturally isomorphic to
$(\gr S)\otimes_K F$, so we identify these two algebras. Since $K$
is finitely generated over $\Z$, $F$ is a finite field.

Since $\gr S$ is free over $K$, there is a $K$-basis of
$\gr S$, say,
\begin{equation}
\label{4.6.1}
\{1\}\cup \{x_i\}\cup \{ {\text{higher degree terms}}\},
\end{equation}
where $\bigoplus_i kx_i$ generates $S$ as an algebra.
We use the same symbols for a $K$-basis of $S$ by lifting, and also for
an $F$-basis of $S\otimes_K F$ (as $S\otimes_K F$ is free over $F$),
and for an $F$-basis of $(\gr S)\otimes_K F$.

(1) Proceed by contradiction and suppose there is a non-affine
automorphism $g\in \Aut(S)$. Then we have
\[
g(x_i)=a_i +\sum_{i'} b_{ii'} x_{i'}+\sum_j c_{ij} y_j,
\]
where $a_i, b_{ii'}, c_{ij}\in K$, some $c_{i_0j_0}\neq 0$, and $y_j$
are basis elements in \eqref{4.6.1} with degree at least 2. Let $K'$
be the localization $K[c_{i_0j_0}^{-1}]$ and let $F$ be a quotient
field of $K'$. Since $K'$ is finitely generated over $\Z$, $F$ is a
finite field. This implies that the composition $K\to K'\to F$ is
surjective and $F$ is a quotient field of $K$. Note that $g\otimes_K
F$ is an automorphism of $S\otimes_K F$. Since $c_{i_0j_0}\neq 0$ in
$F$, $g\otimes_K F$ is not affine, contradicting hypothesis.
Therefore the assertion follows.

(2) Proceed by contradiction and suppose there is a non-triangular
automorphism $h\in \Aut(S[t])$. Then there is an $i$
such that
\[
h(x_i)=\sum_{j\geq 0} f_j t^j,
\]
where $f_j\in S$ and $f_n\neq 0$ for some $n>0$.
Writing $\{z_s \}_s$ for the basis given in \eqref{4.6.1}, write
$f_n=\sum_s c_s z_s$ for some $c_s\neq 0$. Let $K'$
be the localization $K[c_s^{-1}]$ and let $F$ be a quotient
field of $K'$. Since $K'$ is finitely generated over $\Z$, $F$ is a
finite field. This implies that the composition $K\to K'\to F$ is
surjective and $F$ is a quotient field of $K$. Note that $h\otimes_K
F$ is an automorphism of $S\otimes_K F[t]$. Since $c_s\neq 0$ in
$F$, $h\otimes_K F$ is not triangular, contradicting hypothesis.
Therefore the assertion follows.

(3) The proof is similar and omitted.
\end{proof}

\subsection{Factor rings}
\label{yysec4.3}

In this subsection we assume that $A$ is filtered algebra with filtration
$\{F_i A\}_{i\geq 0}$ such that the associated graded algebra is
a domain. Let $Y=\bigoplus_{i=1}^n kx_i$ be a submodule of $F_1 A$
such that $F_1 A=Y\oplus k$. Assume that $A$ is finitely generated free
over its center $R$. Let $I$ be an ideal of $R$ and let $\overline{\cdot}$
denote the factor map $R\to R/I=: \overline{R}$  and the factor map
$A\to A/I=: \overline{A}$.

\begin{proposition}
\label{yypro4.7} Retain the above notations. Suppose that
\begin{enumerate}
\item
$Y\cong \overline{Y}$,
\item
the center of $\overline{A}$ is $\overline{R}$.
\item
the associated graded ring $\gr \overline{A}$ is a domain.
\end{enumerate}
Then $\overline{A}$ is finitely generated free over $\overline{R}=C(\overline{A})$ and
$d(\overline{A}/\overline{R})=\overline{d(A/R)}$.
As a consequence, if $d(A/R)$ is $(-s)$-dominating, so is
$d(\overline{A}/\overline{R})$.
\end{proposition}

\begin{proof} Since $\overline{A}\cong A\otimes_R \overline{R}$,
$\overline{A}$ is finitely generated free over $\overline{R}$:
we may use the $R$-free basis of $A$
for the $\overline{R}$-free basis of $\overline{A}$. Then
$\tr(\overline{f})=\overline{\tr(f)}$ for all $f \in A$, and consequently
$d(\overline{A}/\overline{R})=\overline{d(A/R)}$.
The last assertion follows from the fact $\gr \overline{A}$
is a domain.
\end{proof}

In general if $d(\overline{A}/\overline{R})$ is $(-s)$-dominating,
$d(A/R)$ may not be $(-s)$-dominating. Consider the following example.

\begin{example}
\label{yyex4.8}
Let $A$ be the algebra $k\langle x,y\rangle/(y^2x-xy^2, yx^2-x^2y)$.
Then the center $R$ of $A$ is generated by $x^2,y^2$ and $z:=xy+yx$,
and the discriminant $d(A/C(A))=(xy-yx)^4$. It is easy to check
that $(xy-yx)^4$ is not dominating in $A$.

Let $\overline{A}$ be
the algebra $A/(x^6-y^2)$, which is studied in
\cite[Example 5.8]{CPWZ}. By Proposition \ref{yypro4.7},
$d(\overline{A}/\overline{R})=\overline{d(A/R)}=(xy-yx)^4$
which can be written as $(z-2x^4)^2(z+2x^4)^2$ in
$\overline{R}$. By the analysis in \cite[Example 5.8]{CPWZ}
which uses a non-standard filtration determined by $\deg x=1$
and $\deg y=3$, $(z-2x^4)^2(z+2x^4)^2$ is dominating.
\end{example}

\subsection{Discriminants of filtered algebras}
\label{yysec4.4}

Let $\Lambda$ be a totally ordered abelian
semigroup (e.g., $\N^n$ with the left lexicographic ordering).
We say $B$ is a $\Lambda$-filtered algebra if there is a
filtration $F=\{F_g B\mid g\in \Lambda\}$ such that
$B=\bigcup_{g\in \Lambda} F_g B$.
The associated graded algebra is defined to be
\[
\gr_F B=\bigoplus_{g\in \Lambda} F_g B/F_{<g} B,
\]
where $F_{<g}=\sum_{h<g} F_h B$.  For every nonzero $f\in B$,
we can define the degree of $f$ to be the degree of
$\gr f$ in $\gr_F B$.

We do not assume that $\gr B$ is connected graded,
even if $\Lambda=\N$. Inductively, we identify
the $k$-module $B_g$ with the graded $k$-module
$\bigoplus_{h\leq g} (\gr B)_h$ (with some choices)
so that taking the principal term of $f$,
denoted by $\gr(f)$, can be realized as a projection $B_g\to (\gr B)_h$
if $f\in B_h\setminus B_{<h}$. So $B$ is identified with
$\bigoplus_{g\in \Lambda}(\gr B)_g$ as a $k$-module, and we use
$\xi: \gr B\to B$ denote the inverse of this identification map.
By using $\xi$, elements in $\gr B$ can be viewed as elements
in $B$. Two elements $f$ and $g$ in $B$ or in $\gr B$ are said to be
\emph{$\lambda$-equivalent} if both $\deg f$ and $\deg g$ are no more
than $\lambda$ and $\deg (f-g)<\lambda$. In this case we write
$f\equiv_\lambda g$.

Let $C$ be the center (or more generally, a central subalgebra)
of $B$ such that $B$ is finitely generated free over $C$ with a
basis $b=\{b_1=1,b_2, \cdots,b_w\}$. It is clear that $R:=\gr C$
is a central  subalgebra of $\gr B$. Let $\gr b$ denote the
set $\{\gr b_1,\dots, \gr b_w\}$. Suppose that
\begin{equation}\label{4.8.1}
{\text{$\gr B$ is finitely generated free over $\gr C$ with a
basis $\gr b$.}}
\end{equation}
Note that in general, even if $C$ is the center of $B$, \eqref{4.8.1}
could fail. The following lemma is easy.

\begin{lemma}
\label{yylem4.9}
Assume \eqref{4.8.1} and let $\lambda,\lambda'\in \Lambda$.
The following hold.
\begin{enumerate}
\item
If $\deg (f)=\lambda\geq \deg(g)=\lambda'$, then
$\gr(fg)\equiv_{\lambda+\lambda'}
\gr(f)\gr(g)$, $\gr(af)=a\gr f$ for $a\in k$, and
$\gr(f+g)\equiv_{\lambda} \gr(f)+gr(g)$.
\item
If $\deg f\leq \lambda$, then
$\tr(\gr f)\equiv_{\lambda} \gr \tr(f)$.
\end{enumerate}
\end{lemma}

\begin{proof} (1) Clear.

(2)  It suffices to show the assertion when $\lambda=\deg f$.
By \eqref{4.8.1}, $f b_i=\sum_j r_{ij} b_j$ for some $r_{ij}\in R$ and
$\deg r_{ij}b_j\leq \deg fb_i=:\phi$. Then
\[
\gr(f) \gr(b_i)=\gr(fb_i)
\equiv_{\phi} \sum_j \gr(r_{ij}b_j)
\equiv_{\phi} \sum_j \gr(r_{ij}) \gr (b_j)
\]
with $\deg \gr(r_{ii})\leq \phi-\deg (b_i)=\lambda$. Hence
$\tr(\gr(f))\equiv_{\lambda} \sum_i \gr(r_{ii})$
\end{proof}

\begin{proposition}
\label{yypro4.10}
Retain the above notation and
assume \eqref{4.8.1}. If $d_{w}(\gr B/R)$ is nonzero, then
$\gr d_w(B/C)=d_w(\gr B/R)$.
\end{proposition}

\begin{proof} Since $\gr b$ is a basis of $\gr B$ over $R$, $d_w(\gr B/R)$
is homogeneous of degree $N:=2 \sum_{i=1}^w \deg (\gr b_i)$ [Lemma
\ref{yylem2.6}].  Let $\sigma$ be in $S_w$.
By Lemma \ref{yylem4.9}(2), $\deg \tr(b_i b_{\sigma(i)})
\leq \deg b_i+\deg b_{\sigma(i)}$, so $\deg \prod_{i=1}^w
\tr(b_i b_{\sigma(i)})\leq N$. Now we compute:
\[
\begin{aligned}
\gr d_w(B/C)&=\gr\;  [\det (tr(b_ib_j))]
=\gr\; [\sum_{\sigma\in S_w} (-1)^{|\sigma|} \prod_{i=1}^w
\tr(b_i b_{\sigma(i)})]\\
&\equiv_N \sum_{\sigma\in S_w} (-1)^{|\sigma|}
\gr [\prod_{i=1}^w \tr(b_i b_{\sigma(i)})]\\
&\equiv_N \sum_{\sigma\in S_w} (-1)^{|\sigma|}
\prod_{i=1}^w \gr [\tr(b_i b_{\sigma(i)})]\\
&\equiv_N \sum_{\sigma\in S_w} (-1)^{|\sigma|}
\prod_{i=1}^w \tr(\gr[b_i b_{\sigma(i)}])\\
&\equiv_N \sum_{\sigma\in S_w} (-1)^{|\sigma|}
\prod_{i=1}^w \tr(\gr(b_i)\gr( b_{\sigma(i)}))\\
&\equiv_N d_w(\gr B/R).
\end{aligned}
\]
The assertion follows.
\end{proof}

\subsection{Locally nilpotent derivations}
\label{yysec4.5}

As in the previous subsection, let $\Lambda$ be a totally ordered
abelian semigroup and let $B$ be a finitely generated
$\Lambda$-filtered algebra. Let $\partial$ be a derivation of $B$. Let
$X$ be a set of generators of $B$ as a $k$-algebra.  Define the degree
of $\partial$, denoted by $\deg \partial$, to be the maximal element
of $\deg \partial (x)-\deg x$ for all $x\in X$ (to construct $\deg
\partial (x) - \deg x$, one may have to pass to a totally ordered
abelian group containing $\Lambda$). By the Leibniz rule, $\deg
\partial (f)\leq \deg \partial +\deg f$ for all $f\in B$.  Suppose
$\deg \partial\in \Lambda$ exists, and define $\gr \partial$ by
\[
(\gr \partial) (\gr f)=\begin{cases} 0 & \deg \partial (f)<\deg
\partial+\deg f\\
\gr (\partial(f)) &\deg \partial (f)=\deg
\partial+\deg f\end{cases}
\]
for all $\gr f\in \gr_F B$. It is easy to see that this definition
is independent of the choice of $f\in B$. The following lemma
is not hard and the proof is omitted.

\begin{lemma}
\label{yylem4.11} Let $\Lambda$ be a totally ordered abelian
semigroup and $B$ be a finitely generated $\Lambda$-filtered algebra.
If $\partial$ is a nonzero derivation, then $\gr \partial$ is a
nonzero homogeneous derivation of degree $\deg \partial$.
If $\partial$ is locally nilpotent, then so is $\gr \partial$.
\end{lemma}

\section{$q$-quantum Weyl algebras}
\label{yysec5}

Fix $q\in k^{\times}$ and let $A_q=k\langle x,y\rangle /(yx=qxy+1)$.
If $q=1$, $A_1$ is the usual
first Weyl algebra. In this section we assume that $q\neq 1$. When $q=-1$,
the automorphism group of $A_{-1}$ was studied in
\cite{CPWZ}. If $q\neq \pm 1$, it is well-known that
$\Aut(A_q)=k^\times$ \cite{AlD}. The purpose of this section is not
to give another proof this result, but to compute the discriminant
of this algebra, in order to describe the automorphism group of other
related algebras (such as the tensor product of $A_q$'s).

Suppose $q$ is a primitive $n$th root of unity for some $n\geq 2$. 
In keeping with the notation in previous sections,
let $B=A_q$. We consider $B$ as an $\N$-filtered algebra with
$\deg x=1$ and $\deg y=0$. The following lemma is easy
to check. We identify $x$ and $y$ with $\gr x$ and $\gr y$ in
$\gr B$.

\begin{lemma}
\label{yylem5.1} Retain the above notation and
let $q$ be a primitive $n$th root of unity for some $n\geq 2$.
\begin{enumerate}
\item
$B$ is an $\N$-filtered algebra with $\deg x=1$ and $\deg y=0$
such that $\gr B=k_q[x,y]$ with $\deg x=1$ and $\deg y=0$.
\item
The center of $B$ is $C:=k[x^n,y^n]$. Let $R=\gr C$. Then $R$,
which is the polynomial subalgebra $k[x^n,y^n]$ of $k_q[x,y]$, is the
center of $\gr B$.
\item
There is a subset $b=\{x^iy^j\mid 0\leq i,j \leq n\}\subset B$ such that
$B$ is a finitely generated free module over $C$ with the basis $b$.
\item
$\gr B$ is a finitely generated free module over $R=\gr C$ with the
basis $\gr b$.
\item
The condition \eqref{4.8.1} in Subsection {\rm{\ref{yysec4.4}}} holds.
\end{enumerate}
\end{lemma}

\begin{proposition}
\label{yypro5.2} Suppose $q$ is a primitive $n$th root of
unity with $n \geq 2$. Then
\[
d(A_q/C(A_q))=_{k^\times} (x^ny^n)^{n(n-1)} +\cwlt.
\] 
As a consequence, $d(A_q/C(A_q))$ is dominating.
\end{proposition}

\begin{proof} Retain the notation in Lemma \ref{yylem5.1},
let $B=A_q$, $C=C(B)$ and $R=C(\gr B)$.
By Proposition \ref{yypro2.8} (with $r=w=n^2$),
$d_w(\gr B/R)=(x^ny^n)^{n(n-1)}$. By Proposition \ref{yypro4.10},
$\gr d_w(B/C)=(x^ny^n)^{n(n-1)}$. In particular,
$d_w(B/C)\neq 0$. Write $d:=d_w(B/C)=(x^ny^n)^{n(n-1)}+
\sum_{i,j}a_{i,j} x^i y^j$ with $a_{ij}\in k$.
Then the equation $d_w(\gr B/R)=(x^ny^n)^{n(n-1)}$ implies that
\[
d=(y^{n^2(n-1)}) x^{n^2(n-1)}+\sum_{i<n^2(n-1)} x^i
(\sum_j a_{ij} y^j).
\]
This means that if $a_{ij}\neq 0$, then $i<n^2(n-1)$. By symmetry
(or using a different filtration of $B$), one sees that
if $a_{ij}\neq 0$, then $j<n^2(n-1)$. Thus the assertion
follows.
\end{proof}

Based on computer calculations, we make the following conjecture.

\begin{conjecture}
\label{yycon5.3} Suppose $q$ is a primitive $n$th root of
unity. Then
\[
d(A_q/C(A_q))=_{k^\times } ((1-q)^nx^ny^n-1)^{n(n-1)}.
\]
\end{conjecture}

This conjecture holds when $n=2$: see \cite[Example 1.7(1)]{CPWZ}.

For the rest of this section we consider the tensor
product of $q$-quantum Weyl algebras. Use the letter $B$ for the
tensor product $A_{q_1}\otimes \cdots \otimes A_{q_m}$. The following
corollary follows immediately from Proposition \ref{yypro5.2}
and \cite[Theorem 5.5]{CPWZ}.

\begin{corollary}
\label{yycor5.4}
Let $B=A_{q_1}\otimes \cdots \otimes A_{q_m}$ and assume that
each $1\neq q_i$ is a root of unity. Then $B$ is in $\Af$,
namely, $d(B/C(B))$ is dominating. As a consequence,
$\Aut(B)$ is affine.
\end{corollary}

From now on we do not assume that the parameters $q_i$ are roots of unity.
Here is the first part of Theorem \ref{yythm0.2}.

\begin{theorem}
\label{yythm5.5} Let $B=A_{q_1}\otimes \cdots \otimes A_{q_m}$ be
defined as before. Assume that $q_i\neq 1$ for all $i=1,\dots,m$.
Then every algebra automorphism of $B$ is affine.
\end{theorem}

\begin{proof} Let $Y$ be the subspace $\bigoplus_{i=1}^m
(kx_i\oplus ky_i)$. Then $Y$ is a generating space of $B$ and
$B$ is a filtered algebra with standard filtration defined by
$F_n B=(Y\oplus k)^n$ (and with $\deg x_i=\deg y_i=1$ for all $i$).
Clearly, $\gr B$ is a skew polynomial ring. So we have a monomial
basis for the algebra $B$.

Proceed by contradiction and assume that there is an automorphism
$g$ of $B$ which is not affine. Write
$g(x_i), g(y_i), g^{-1}(x_i), g^{-1}(y_i)$ as linear combinations
of the monomial basis, and let $K$ be the $\Z$-subalgebra of
$k$ generated by the collection of the nonzero coefficients $\{c_w\}_w$
of $g(x_i), g(y_i), g^{-1}(x_i), g^{-1}(y_i)$, along with $\{c_w^{-1}\}_w$,
$\{q_i^{\pm 1}\}_i$ and $\{(q_i-1)^{-1}\}_i$. (If $k$ is not a field,
adjoin inverses as necessary.) Let $S$ be the
$K$-subalgebra of $B$ generated by $\{x_i\}_{i=1}^m \bigcup
\{y_i\}_{i=1}^m$. By the definition of $K$,
both $g$ and $g^{-1}$ are well-defined as algebra homomorphisms
of $S$. Since $g\circ g^{-1}$ and $g^{-1}\circ g$ are the identity
when restricted to the $K$-subalgebra $S\subset B$, $g$ is an
automorphism of $S$ with inverse $g^{-1}$. Since the relations
of $B$ (and of $S$) are of the form
\begin{equation}
\label{5.5.1}
y_i x_i= q_i x_i y_i+1, \quad [x_i,x_j]=[x_i,y_j]=[y_i,y_j]=0
\end{equation}
for all $i\neq j$, one can check that $\gr S$ is a skew polynomial algebra
with base ring $K$ (or $S$ is an iterated Ore extension starting with $K$).
In fact, it is free over $K$. Hence the hypotheses of Lemma
\ref{yylem4.6} hold.

Now consider a finite quotient field $F=K/\mathfrak{m}$. Then
the image of $q_i$, denoted by $\bar{q}_i$, is not 1 in $F$.
Since $S$ is an iterated Ore extension, $S\otimes_K F$
is also an iterated Ore extension with the relation
\eqref{5.5.1} with $q_i$ being replaced by $\bar{q}_i$.
Therefore $S\otimes_K F$ is isomorphic to the product
of quantum Weyl algebras $A_{\bar{q}_i}$ over the field $F$,
where $\bar{q}_i\neq 1$. Since $F$ is a finite field, $\bar{q}_i$
is a root of unity. By Corollary \ref{yycor5.4}
$\Aut(S\otimes_K F)=\Aut_{\af}(S\otimes_K F)$. 
By Lemma \ref{yylem4.6}(1),
$\Aut(S)=\Aut_{\af}(S)$, which contradicts the fact that
$g|_S$ is not affine. The assertion follows.
\end{proof}

To prove the rest of Theorem \ref{yythm0.2} (and Theorem
\ref{yythm5.7} below), we need the following lemma.

\begin{lemma}
\label{yylem5.6} Let $B=A_{q_1}\otimes \cdots \otimes A_{q_m}$
with $q_i\neq 1$ for all $i$. Let $Y$ be the subspace $\bigoplus_{i=1}^m
(kx_i\oplus ky_i)$. Let $g$ be a
(necessarily affine) automorphism of $B$.
\begin{enumerate}
\item
$g(Y)=Y$.
\item
For each $i$, either $g(x_i)=b_i x_{i'}$
and $g(y_i)=f_i y_{i'}$ for some $i'$, or $g(x_i)=c_i y_{i'}$
and $g(y_i)=e_i x_{i'}$ for some $i'$.
\item
If $k$ is a field, then $\Aut_{\af}(B)$ is an algebraic group that fits
into the exact sequence
\[
1\to (k^\times)^m\to \Aut_{\af}(B)\to S\to 1,
\]
where $S$ is the finite group generated by all automorphisms
$g$ of the form $g(x_i)=x_{i'}$ and $g(y_i)=y_{i'}$ for some $i'$,
or $g(x_i)=y_{i'}$ and $g(y_i)=x_{i'}$ for some $i'$.
\item
If $q_i\neq q_j^{-1}$ for all $i,j$, then there is a permutation
$\sigma\in S_m$ and $b_i\in k^\times$ such that $g(x_i)=b_i x_{\sigma(i)}$
and $g(y_i)=b_i^{-1} y_{\sigma(i)}$ for all $i$. Further
$q_i=q_{\sigma(i)}$ for all $i$.
\item
If $q_i\neq \pm 1$ and $q_i\neq q_j^{\pm 1}$ for all $i\neq j$, then
$\Aut_{\af}(B)=(k^\times)^m$.
\item
If $q_i=q\neq \pm 1$ for all $i$, then $\Aut_{\af}(B)=S_m\rtimes (k^\times)^m$.
\end{enumerate}
\end{lemma}

\begin{proof} (1) Write
\begin{align*}
g(x_i)&=a_i+ \sum_{s=1}^n b_{is} x_s+\sum_{t=1}^n c_{it}y_t=a_i+X_i,\\
g(y_i)&=d_i+ \sum_{s=1}^n e_{is} x_s+\sum_{t=1}^n f_{it}y_t=d_i+Y_i,
\end{align*}
where $a_i, b_{is}, c_{it}, d_i, e_{is}, f_{it}\in k$.
Applying $g$ to the relation $1=y_i x_i-q x_iy_i$ (where we write $q=q_i$), we have
\begin{align*}
1&=g(y_i)g(x_i)-qg(x_i)g(y_i)\\
&=Y_iX_i+a_iY_i+d_iX_i+a_id_i-q(X_iY_i+a_iY_i+d_iX_i+a_id_i)\\
&=Y_iX_i-qX_iY_i+(1-q)[a_i Y_i+d_i X_i+a_id_i].
\end{align*}
By using the relations of $B$, the degree 1 part of the above equation
is
\[
0=(1-q)[a_i Y_i+d_i X_i].
\]
Since $q\neq 1$, $a_i Y_i+d_i X_i=0$.
If $a_i$ or $d_i$ is nonzero, then $X_i$ and $Y_i$ are linearly
dependent, which contradicts the fact that $\{1,x_i,y_i\}$
is linearly independent. Therefore $a_i=d_i=0$ for all $i$.
The assertion follows.

(2) We keep the notation from part (1), and we know that $a_i=d_i=0$
for all $i$. Note that the $x_s$'s commute and the $y_t$'s commute. Then
\begin{align*}
1&=g(y_i)g(x_i)-q_ig(x_i)g(y_i)=Y_iX_i-q_iX_iY_i\\
&=\left(\sum_{s=1}^n e_{is} x_s+\sum_{t=1}^n f_{it}y_t\right)
\left(\sum_{s=1}^n b_{is} x_s+\sum_{t=1}^n c_{it}y_t\right)\\
&\qquad
-q_i\left(\sum_{s=1}^n b_{is} x_s+\sum_{t=1}^n c_{it}y_t\right)
\left(\sum_{s=1}^n e_{is} x_s+\sum_{t=1}^n f_{it}y_t\right)\\
&=(1-q_i)\left[ \left(\sum_{s=1}^n e_{is} x_s\right) \left(\sum_{s=1}^n b_{is} x_s\right)
+\left(\sum_{t=1}^n f_{it}y_t\right) \left(\sum_{t=1}^n c_{it}y_t\right)\right] \\
&\qquad + \left(\sum_{s=1}^n e_{is} x_s\right) \left(\sum_{t=1}^n c_{it}y_t\right)
+\left(\sum_{t=1}^n f_{it}y_t\right) \left(\sum_{s=1}^n b_{is} x_s\right)\\
&\qquad -q_i\left(\sum_{s=1}^n b_{is} x_s\right) \left(\sum_{t=1}^n f_{it}y_t\right)
-q_i\left(\sum_{t=1}^n c_{it}y_t\right) \left(\sum_{s=1}^n e_{is} x_s\right).
\end{align*}
By using the monomial basis of $B$, one sees that
\[
\left(\sum_{s=1}^n e_{is} x_s\right) \left(\sum_{s=1}^n b_{is} x_s\right)=
\left(\sum_{t=1}^n f_{it}y_t \right) \left(\sum_{t=1}^n c_{it}y_t\right)=0.
\]
Since $B$ is a domain, we have either
\[
\sum_{s=1}^n e_{is} x_s=0=\sum_{t=1}^n c_{it}y_t
\]
or
\[
\sum_{s=1}^n b_{is} x_s=0=\sum_{t=1}^n f_{it}y_t.
\]
In the first case, the equation becomes
\begin{align*}
1&=\sum_{s\neq t}b_{is}f_{it}(1-q_i)x_sy_t+\sum_sb_{is}f_{is}
[(q_s-q_i)x_sy_s+1].
\end{align*}
This implies that $b_{is}f_{it}=0$ for all $s\neq t$.
As a consequence, $b_{is}$ is zero except for one $s$ and
$f_{it}$ is zero except for one $t$. The assertion follows.
The argument for the second case is similar.

(3) This follows from part (2).

(4) Suppose that $g(x_i)=c_iy_{i'}$ and $g(y_i)=e_i x_{i'}$ for some $i'$.
Applying $g$ to $1=y_ix_i-q_i x_iy_i$ we have
\[
1=c_ie_i (x_{i'}y_{i'}-q_iy_{i'}x_{i'})=
 c_ie_i [(1-q_iq_{i'})x_{i'}y_{i'}-q_i].
\]
which implies that $q_i=q_{i'}^{-1}$, a contradiction. By part (2),
we have that for each $i$, $g(x_i)=b_ix_{i'}$ and $g(y_i)=f_i y_{i'}$.
Further, by the relation, one has that $q_i=q_{i'}$ and $f_i=b^{-1}_i$.
The assignment $i\mapsto i'$ defines the required permutation $\sigma$.
Finally it is easy to check that $q_i=q_{\sigma(i)}$ for all $i$.

(5,6) Follows from part (4).
\end{proof}

\begin{theorem}
\label{yythm5.7} Let $B=A_{q_1}\otimes \cdots \otimes A_{q_m}$.
Assume that $q_i\neq 1$ for all $i$. Then
\begin{enumerate}
 \item
Every automorphism of $B$ is affine. As a consequence, the
following hold.
\begin{enumerate}
\item
If $q_i\neq \pm1$ and $q_i\neq q_j^{\pm 1}$ for all $i\neq j$, then
$\Aut(B)=(k^\times)^m$.
\item
If $q_i=q\neq \pm 1$ for all $i$, then $\Aut(B)=S_m \rtimes (k^\times)^m$.
\end{enumerate}
\item
The automorphism group of $B[t]$ is triangular.
\item
If $k$ is a field, then $\Aut(B)$ is an algebraic group that fits
into the exact sequence
\[
1\to (k^\times)^m\to \Aut(B)\to S\to 1
\]
for some finite group $S$.
\item
If $\Z\subset k$, $\LNDer(B)=\{0\}$.
\end{enumerate}
\end{theorem}

\begin{proof}
The main assertion in part (1) is Theorem \ref{yythm5.5}.
Parts(a,b) follow from
the main assertion and Lemma \ref{yylem5.6}(5,6).

The proof of part (2) is similar to the proof of Theorem
\ref{yythm5.5} and omitted.

(3) This is a consequence of part (1) and Lemma \ref{yylem5.6}(3).

(4) By localizing $k$, we may assume that $\Q \subseteq k$.
Then this is a consequence of part (2) and \cite[Lemma 3.3(2)]{CPWZ}.
\end{proof}

Theorem \ref{yythm0.2} follows from Theorem \ref{yythm5.7}.

\subsection*{Acknowledgments}

The authors would like to thank Ken Goodearl, Colin
Ingalls, Rajesh Kulkarni, and Milen Yakimov for several
conversations on this topic during the Banff workshop in
October 2012 and the NAGRT program at MSRI in the Spring of 2013.
S. Ceken was supported by the Scientific and Technological Research
Council of Turkey (TUBITAK), Science Fellowships and Grant Programmes
Department (Programme no. 2214). Y.H. Wang was supported by the
Natural Science Foundation of China (grant no. 10901098 and 11271239).
J. J. Zhang was
supported by the US National Science Foundation (NSF grant No. DMS
0855743).

\providecommand{\bysame}{\leavevmode\hbox to3em{\hrulefill}\thinspace}
\providecommand{\MR}{\relax\ifhmode\unskip\space\fi MR }
\providecommand{\MRhref}[2]{%

\href{http://www.ams.org/mathscinet-getitem?mr=#1}{#2} }
\providecommand{\href}[2]{#2}


\begin{thebibliography}{10}

\bibitem[AlC]{AlC}
J. Alev and M. Chamarie,
D{\'e}rivations et Automorphismes de Quelques Alg{\'e}bres Quantiques,
\emph{Comm. Algebra}, {\bf 20}(6) (1992), 1787--1802.

\bibitem[AlD]{AlD}
J. Alev and F. Dumas,
Rigidit{\'e} des plongements des quotients primitifs minimaux de
$U_q(sl(2))$ dans l'alg{\'e}bre quantique de Weyl-Hayashi,
\emph{Nagoya Math. J.} {\bf 143} (1996), 119--146.

\bibitem[AD]{AnD}
N. Andruskiewitsch and F. Dumas,
On the automorphisms of $U^+_q (g)$, Quantum groups, 107--133,
IRMA Lect. Math. Theor. Phys., 12, Eur. Math. Soc., Z{\"u}rich, 2008.

\bibitem[AW]{AW}
S. Alaca and K.S. Williams, \emph{Introductory Algebraic Number
Theory}, Cambridge Univ. Press, Cambridge, 2004.

\bibitem[BJ]{BJ}
V.V. Bavula and D.A. Jordan,
Isomorphism problems and groups of automorphisms for generalized Weyl
algebras. \emph{Trans. Amer. Math. Soc.} {\bf 353} (2001), no. 2, 769--794.

\bibitem[CPWZ]{CPWZ}
S. Ceken, J. Palmieri, Y.-H. Wang and J.J. Zhang,
The discriminant controls automorphism groups of noncommutative
algebras, preprint (2013).

\bibitem[GTK]{GTK}
J. G{\'o}mez-Torrecillas and L. El Kaoutit,
The group of automorphisms of the coordinate ring of quantum
symplectic space, Beitr{\"a}ge Algebra Geom. {\bf 43} (2002),
no. 2, 597--601.

\bibitem[GY]{GY}
K. R. Goodearl and M. T. Yakimov,
Unipotent and Nakayama automorphisms of quantum nilpotent algebras,
preprint (2013), arXiv:1311.0278.

\bibitem[Ju]{Ju}
H. W. E. Jung, {\"U}ber ganze birationale Transformationen der Ebene,
\emph{J. reine angew. Math.},
{\bf 184} (1942), 161--174.

\bibitem[KKZ1]{KKZ1}
E. Kirkman, J. Kuzmanovich, and J.J. Zhang,
Rigidity of graded regular algebras,
\emph{Trans. Amer. Math. Soc.} {\bf  360} (2008), no. 12, 6331--6369.

\bibitem[KKZ2]{KKZ2}
E. Kirkman, J. Kuzmanovich and J.J. Zhang,
Gorenstein subrings of invariants under
Hopf algebra actions,
\emph{J. Algebra} {\bf 322} (2009), no. 10, 3640--3669.

\bibitem[KKZ3]{KKZ3}
E. Kirkman, J. Kuzmanovich and J.J. Zhang,
Shephard-Todd-Chevalley theorem for
skew polynomial rings,
\emph{Algebr. Represent. Theory} {\bf 13}  (2010), no. 2, 127--158

\bibitem[MR]{MR}
J.C. McConnell and J.C. Robson,
\emph{Noncommutative Noetherian Rings}, Wiley, Chichester, 1987.

\bibitem[Re]{Re}
I. Reiner,
\emph{Maximal orders}, London Mathematical Society Monographs.
New Series, 28. The Clarendon Press, Oxford University Press, Oxford, 2003.

\bibitem[SU]{SU}
I. Shestakov and U. Umirbaev,
The tame and the wild automorphisms of polynomial rings in
three variables, \emph{J. Amer. Math. Soc.} {\bf 17} (1) (2004)
197--227.

\bibitem[SVdB]{SVdB}
J. T. Stafford and M. Van den Bergh,
Noncommutative resolutions and rational singularities,
\emph{Michigan Math. J.}, {\bf 57}, (2008), 659--674.

\bibitem[SAV]{SAV}
M. Su{\'a}rez-Alvarez and Q. Vivas,
Automorphisms and isomorphism of quantum generalized Weyl algebras,
preprint (2012), arXiv:1206.4417v1.

\bibitem[St]{St}
W. A. Stein, \emph{Algebraic Number Theory: A Computational Approach},
preprint, \url{http://wstein.org/books/ant/}.

\bibitem[vdK]{vdK}
W. van der Kulk, On polynomial rings in two variables,
\emph{Nieuw Archief voor Wisk.} (3) {\bf 1}
(1953), 33--41.


\bibitem[Y1]{Y1}
M. Yakimov, Rigidity of quantum tori and the Andruskiewitsch-Dumas conjecture,
\emph{Selecta Math.} (2013) (to appear),
preprint, arXiv:1204.3218.

\bibitem[Y2]{Y2}
M. Yakimov, The Launois-Lenagan conjecture,
\emph{J. Algebra} \textbf{392} (2013), 1--9,
arXiv:1204.5440.

\end{thebibliography}
\end{document}